\documentclass{article}

\usepackage[ngerman,english]{babel}
\usepackage{amsmath, amssymb,amsthm}
\usepackage{bm}
\usepackage{graphicx}       
\usepackage{comment} 
\usepackage{subfig}
\usepackage{verbatim}
\usepackage{color}

\newtheorem{thm}{Theorem}
\newtheorem{lem}[thm]{Lemma}

\newtheorem{rem}[thm]{Remark}
\newtheorem{defn}[thm]{Definition}

\newcommand{\R}{{\mathbb R}}

\newcommand{\argmin}{\mathop{\rm argmin}}%

\newcommand{\wto}{\rightharpoonup}

\newcommand{\cL}{{\mathcal L}}

\newcommand{\cO}{{\mathcal O}}

\newcommand{\bE}{{\mathbf E}}

\newcommand{\eps}{\epsilon}

\newcommand{\wrt}{with respect to }

\title{Optimal control of multiscale systems using reduced-order models}
\author{W.~Zhang, J.C.~Latorre, G.A.~Pavliotis, and C.~Hartmann}
\date{\today}

\begin{document}
\maketitle

%%%%%%%%%%%%%%%%%%%%%%%%%%%%%%%%%%%%%%%%%%%%
%% INTRO
%%%%%%%%%%%%%%%%%%%%%%%%%%%%%%%%%%%%%%%%%%%% 

\begin{abstract}
We study optimal control of diffusions with slow and fast variables and address a question 
raised by practitioners: is it possible to first eliminate the fast variables before solving the 
optimal control problem and then use the optimal control computed from the reduced-order model to
control the original, high-dimensional system? The strategy ``first reduce, then optimize''---rather than ``first 
optimize, then reduce''---is motivated by the fact that solving optimal control problems for high-dimensional
multiscale systems is numerically challenging and often computationally prohibitive. We state sufficient and necessary 
conditions, under which the ``first reduce, then control'' strategy can be employed and discuss when it 
should be avoided. We further give numerical examples that illustrate the ``first reduce, then optmize'' approach and discuss 
possible pitfalls.       
\end{abstract}

\section{Introduction}\label{sec:intro}

Optimal control problems for diffusion processes have attracted a lot of attention in the last decades, both in terms of the development of the theory as well as in terms of concrete applications to problems in the sciences, engineering and finance~\cite{fleming2006,kushner2001}. Stochastic control problems appear in a variety of applications, such as statistics \cite{WangDupuis2004,is_multiscale}, 
financial mathematics \cite{Davis1990,PhamBook}, molecular dynamics \cite{ScWiHa2012,HaSc12} and materials science \cite{Steinbrecher2010,Asplund2011}, to mention just a few. 
A common feature of the models used is that they are high-dimensional and possess several characteristic time scales. For instance, in single molecule alignment experiments, a laser field is used to stabilize the slowly-varying orientation of a molecule in 
solution that is coupled to the fast internal vibrations of the molecule, but ideally the controller would like to base the control protocol only on the relevant slow degree of freedom, i.e.~the orientation of the molecule \cite{Stapelfeldt2004}. 

If the time scales in the system are well separated, it is possible to eliminate the fast degrees of freedom and to derive low-order reduced 
models, using averaging and homogenization techniques~\cite{pavliotis2008}. Homogenization of stochastic control systems has been extensively studied by applied 
analysts using a variety of different mathematical tools, including viscosity solutions of the Hamilton-Jacobi-Bellman equation~\cite{Blankenship1987,Evans1989,Alvarez2002,Lions2003}, backward stochastic differential equations \cite{Buckdahn1998, Buckdahn1999,Ichihara2005}, Gamma convergence \cite{Gamma1987,Gamma1989} and  occupation measures \cite{Kushner1986,kushner1990,Kurtz2001}.  The latter has been also employed to analyse deterministic control systems, together with differential inclusion techniques \cite{Gaitsgory1992,Vigodner1997,Grammel1997,Artstein2002,Watbled2005}. The convergence analysis of multiscale control systems, both deterministic and stochastic, is quite involved and non-constructive, in that the limiting equations of motion are not 
given in explicit or closed form; see \cite{Kokotovic1984,gajic2001,Kabanov2003} for notable exceptions, dealing mainly with the case when the dynamics is linear. We shall refer to all these approaches---without trying to be exhaustive---as ``first optimize, then reduce''. 

On the other side of the spectrum are model order reduction (MOR) techniques for large-scale linear and bilinear control systems that 
are based on tools from linear algebra and rational approximation. MOR aims at approximating the response of a controlled system to 
any given control input from a certain class, e.g., piecewise constant or square integrable functions; see, e.g., \cite{gugercin2004,antoulas2005} and the 
references given there. A very popular MOR method is balanced truncation that gives easily computable error bounds  
in terms of the Hankel norm of the corresponding transfer functions \cite{moore1981,glover1984}, and which has recently been extended to deterministic and 
stochastic slow-fast systems, using averaging and homogenization techniques \cite{balance_hartmann,Ha11,HaSBZu13}. 
In applications MOR is often used to drastically reduce the system dimension, before a possibly computational expensive optimal control problem is solved. In most real-world 
applications, solving an optimal control problems on the basis of the unreduced large-scale model is prohibitive, which explains the popularity
of MOR techniques. We will call this approach ``first reduce, then optimize''.

\subsection{The MOR approach: first reduce, then optimize}

In this paper we focus on optimal control of diffusions with two characteristic time scales. As a representative example, we consider the diffusion of a driven Brownian particle in a two-scale energy landscape in one dimension 
\begin{equation}\label{bm}
dx^{\eps}_{s} = (\sigma u^\eps_{s} -\nabla \Phi (x^{\eps}_{s},x^{\eps}_{s}/\eps))\,ds + \sigma \beta^{-1/2} dw_{s}\,,
\end{equation}
where $u^\eps$ is any time-dependent driving force (or control variable) and $w_t$ is standard one-dimensional Brownian motion. The potential consists of a large metastable part with small-scale superimposed periodic fluctuations, $\Phi(x,y) = \Phi_{0}(x) +p(y)$ with $p(\cdot)$ a 1-periodic function. A typical potential is shown in Figure \ref{fig:multiscalepotential}. 
\begin{figure}
  \includegraphics[width=80mm]{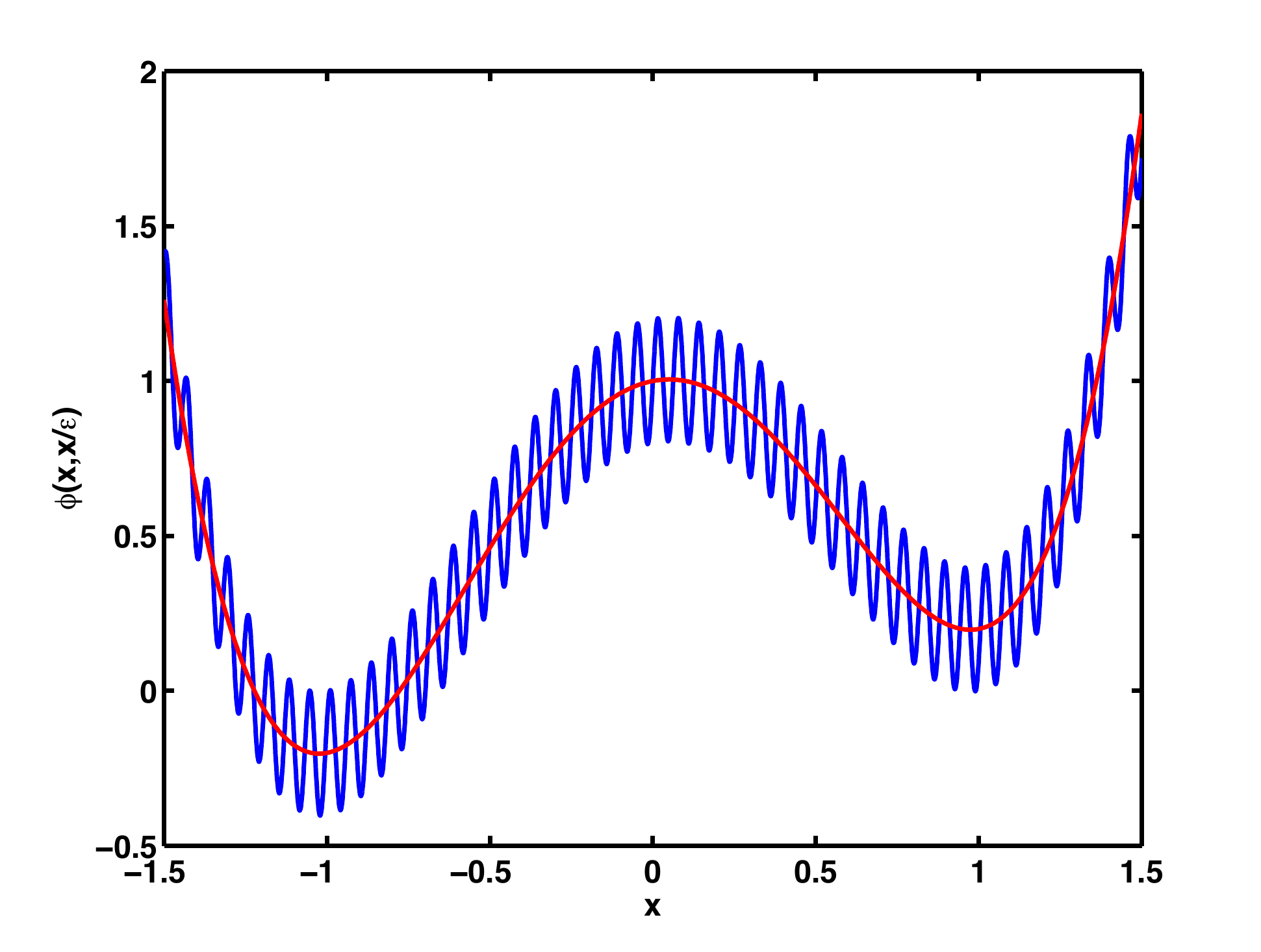}
 \caption{Bistable potential (shown in red) with superimposed small-scale oscillations of period $\eps$ (in blue).} \label{fig:multiscalepotential}
 \end{figure}

Now, if $u^\eps$ is given as a function of time, say bounded and continuous, it is known that $x_{s}^{\eps}$ converges in distribution to a limiting process $x_{s}$ as $\eps\to 0$, where $x_{s}$ solves the homogenized equation~\cite{PavlSt06}
\begin{equation}\label{bmlim}
dx_{s} = (\sigma A u_{s} -A\nabla \Phi_{0}(x_{s}))\,ds + \sqrt{A} \beta^{-1/2} \, dw_{s}\,.
\end{equation}
Here $0<A<1$ is an effective diffusivity that accounts for the slowing down of
the dynamics due to the presence of local minima in the two-scale potential. The property that $x^{\eps}$ weakly converges to $x$ in the sense of probability measures will be referred to as \emph{forward stability} of the homogenized equation. 
Now imagine a situation, in which $u^\eps$ depends on $x_{s}^{\eps}$ via a feedback law
\begin{equation}\label{feedback}
u^{\eps}_{s} = c(x^{\eps}_{s};\eps),
\end{equation}
where $c(\cdot;\eps)$ is a measurable function of $x$. (For simplicity, we do not consider the case that $c$ carries an explicit time-dependence.) Specifically, we choose $u$ from an admissible class of feedback controls so that the cost functional 
\begin{equation*}
J(u^{\eps}) = \bE\left(\int_{0}^{\tau}L(x^{\eps}_{s}, u^\eps_{s})\,ds \right)
\end{equation*}
is minimized for some given running cost $L\ge 0$ associated with the sample paths of $x_{s}^{\eps}$ and $u^{\eps}_{s}$ up to a random stopping time $\tau$ of the process. 

The aim of the paper is to study situations where the cost functional evaluated at $u^{\eps}$, converges to $ J(u)$, with $u$ being the limit of $u^{\eps}$ (in some appropriate sense). Specifically, we are dealing with the situation that 
\[
\inf_{u}J(u^{\eps}) \to \inf_{u} J(u)\,,
\] 
a property that we will refer to as \emph{backward stability}. If the homogenized equation is backward stable, it does not matter whether one first solves the optimal control problem and then sends $\eps$ to $0$ or vice versa, in which case the control $u$ is simply treated as a parameter. One of the implications then is that we can compute optimal controls from the homogenized model, such as (\ref{bmlim}), and use them in the original equation when $\eps$ is sufficiently small. 

Unfortunately very few systems are backward stable in this sense, a notable exception being a system of the form~\eqref{bm} when the running cost $L$ is quadratic in $u$, e.g. \cite[Sec.~4.1]{kushner1990}. The reader may wonder why one should first reduce the equations {\bf before} solving the optimal control problem anyway, rather than the other way round. One answer is that solving optimal control problems for high-dimensional multiscale systems is usually  computationally infeasible, which often leaves no other choice; another answer is that there may be situations, in which a fully resolved model may not be explicitly available, but one only has a sufficiently accurate low-order model that captures the relevant  dynamics of the system. In both cases one wants to make sure that the controls obtained from the low-order reduced model can be used in order to control the original system.   

\subsection{Mathematical justification of the MOR approach}

In this article we consider the exceptional cases of backward stability and give necessary and sufficient conditions under which the reduced systems (disregarding the control) are indeed backward stable. It turns out that a class of optimal control problems that are backward stable are systems that are linear-quadratic in the control variable; they may be nonlinear in the state variables, though, and therefore cover many relevant applications in the sciences and engineering. Moreover we find that an additional requirement is that the controls of the multiscale system converge in a strong sense; an example of weak convergence, in which the systems fails to be backward stable due to lack of sequence continuity, is when the controls are oscillatory with rate $1/\eps$ around its homogenization limit, in case of which $J^{\eps}(u^{\eps})$ does not converge to $J(u)$ unless $J$ is linear in $u$. For a related discussion of weak convergence issues in optimal control, we refer to \cite{AlvarezBardi2007,alvarez2008}. 
Similar problems for parameter estimation and filtering are discussed in \cite{gajic2001,PavlSt06, Sowers-al-2010, Imkeller2012, PapPavSt08}.

Strong convergence of the control is a necessary, but not sufficient condition for backward stability of the model reduction approach (first reduce, then optimize), in which the control variable is treated as a parameter during the homogenization procedure. The class of control problems, which can be homogenized in the above way are systems of SDEs that can be transformed to systems in which the controls are absent. The class of such systems are linear-quadratic in the controls (but possibly nonlinear in the states), and can be transformed by a suitable \emph{logarithmic transformation} of the value function of the optimal control problem: 
\[
  V^{\eps}(x) = \inf_{u^\eps}\bE\left(\int_{0}^{\tau}L(x^{\eps}_{s}, u^\eps_{s})\,ds \middle| x^{\eps}_{0} = x\right)\,.
\]
It can be shown (see \cite{fleming2006}) that the log-transformed value function solves a linear boundary
value problem that does not involve any control variables and can be
homogenized using standard techniques. Once the linear equation has been
homogenized, it can be transformed back to an equivalent optimal control
problem that is precisely the limiting equation of the original multiscale
control problem. A nice feature of the logarithmic transformation approach is that the optimal control can be expressed in terms of the solution of the linear boundary value problem, which can be solved efficiently using Monte-Carlo methods. This approach is helpful when the dynamics are high-dimensional, in which case any grid-based discretization of the above linear boundary value problem is prohibitive. (The case when the stopping time $\tau$ is deterministic and the log-transformed value function solves a linear transport PDE can be treated analogously.)

Our approach is summarized in Table \ref{table:approach}.

\begin{table}[!ht] 
\centering
 \begin{tabular}[h]{ccc}
\vspace{0.5cm}$V^\eps=\min_{u} J^\eps(u)$ & $\underrightarrow{\;\psi^\eps=\exp(-\beta V^\eps)\;}$ &  linear PDE for $\psi^{\eps}$\\
\vspace{0.5cm}$\eps\to 0 \Bigg\downarrow$ & & $\Bigg\downarrow\eps\to 0$\\
$V=\min_{u} J(u)$ & $\underleftarrow{\,\;V=-\beta^{-1}\log\psi \;\,}$ & linear PDE for $\psi$
\end{tabular}\vspace{1cm}

\caption{Schematic approach of the homogenization procedure using logarithmic transformation.} \label{table:approach}
\end{table}

The article is organized as follows: In Section \ref{sec:mms} the model reduction approach for the indefinite time-horizon control problem with multiple time scales is outlined, with a brief introduction to dynamic programming and logarithmic transformations in Section \ref{ssec:logtrafo}. The model reduction problem is illustrated in Section \ref{sec-examples} with three different numerical examples: underdamped motion of Langevin-type (Sec.~\ref{subsec-examples-langevin}), diffusion in a highly-oscillatory potential (Sec.~\ref{subsec-examples-homo}), and the Gaussian linear quadratic regulator (Sec.~\ref{subsec-examples-linear}). The article contains three appendices: Appendix \ref{sec:weak} discusses weak convergence under logarithmic transformations, Appendix \ref{sec:ergode} introduces the infinite time-horizon problem associated with the linear quadratic regulator example, Appendix \ref{sec:bounds} contains the proof of Theorem \ref{thm:entropyIndefinite} and records various identities to bound the cost 
functional and the value function when using suboptimal controls.

%%%%%%%%%%%%%%%%%%%%%%%%%%%%%%%%%%%%%%%%%%%%
%% LOG TRANSFORMS
%%%%%%%%%%%%%%%%%%%%%%%%%%%%%%%%%%%%%%%%%%%% 

\section{Multiscale control problem} \label{sec:mms}

We start by setting the notation which we will use throughout this article. We denote by  
$O\subset\R^{n}$ a bounded open set with sufficiently smooth
boundary $\partial O$. Further let $(z^{\eps,u}_{s})_{s\ge 0}$ be a stochastic process assuming values in $\R^{n}$ that is the solution of     
\begin{equation}\label{sde}
dz^{\eps, u}_{s} = \left( b(z^{\eps, u}_{s};\eps) + \sigma(z^{\eps,
u}_{s};\eps)u^\eps_{s}\right) ds + \sigma(z^{\eps,u}_{s};\eps) \beta^{-1/2} dw_{s}\,,
\end{equation}
where $u^\eps_{s}\in U\subseteq\R^{n}$ is the control applied at time $s$ and
$w=(w_{s})_{s\ge 0}$ is $n$-dimensional Brownian motion and $\beta>0$ is the
(dimensionless) inverse temperature of the system. We assume that, for each
$\eps>0$, drift and noise coefficients, $b(\cdot;\eps)$ and
$\sigma(\cdot;\eps)$, are continuous functions on $\bar{O}$, satisfying the usual
Lipschitz and growth conditions that guarantee existence and uniqueness of the process \cite{oksendal2003}.

\subsection*{Cost functional} We want to control (\ref{sde}) in such a way
that an appropriate cost criterion is minimized where the control is active
until the process leaves the set $O$. Assuming $z_0^{\eps, u} = z\in O$, we define $\tau$ to be the stopping time
\begin{equation}\label{tau}
\tau = \inf\{s>0\,;\,z^{\eps,u}_{s}\notin O\}\,,
\end{equation}
i.e., $\tau$ is the first exit time of the process $z_{s}^{\eps,u}$ from $O$. Our cost criterion reads 
\begin{equation}\label{cost}
  J(u^\eps;z) = \bE\left(\int_0^{\tau} L(z^{\epsilon, u}_s, u^\eps_{s})\,ds
  ~\bigg|~
  z^{\eps, u}_0 = z\right)
\end{equation}
where $L$ is the running cost that we assume to be of the form
\begin{equation}\label{rgcost}
L(z,u) = G(z) + \frac{1}{2}|u|^{2}\,,
\end{equation}
with $G$ being continuous on $\bar{O}$. Note that the $\eps$-dependence of the cost
functional $J$ comes only through the dependence of the control 
on $z^{\eps,u}_{s}$. We will omit the dependence on $z$ in $J(u;z)$ and write it as $J(u)$ whenever there is no ambiguity.  

\subsection{Logarithmic transformation}\label{ssec:logtrafo} 
 
In order to pass to the limit $\eps\to 0$ in (\ref{sde})--(\ref{rgcost}), we resort to the
technique of logarithmic transformations that has been developed by Fleming
and co-workers (see \cite{fleming2006} and the references therein). We start
by recalling the dynamic programming principle for stochastic control problems
of the form (\ref{sde})--(\ref{rgcost}). To this end we make the following
assumptions (see \cite[Secs.~VI.3--5]{fleming2006} for further details on the first two of the following assumptions) :

\subsubsection*{Assumption 1}  For every $\eps>0$, the matrices $a(\cdot;\eps) =\sigma(\cdot;\eps)\sigma(\cdot;\eps)^{T}$ are positive definite with uniformly bounded inverse $a(\cdot;\eps)^{-1}$.

\subsubsection*{Assumption 2}  The running cost $G(z)$ is continuous, 
nonnegative, and $G(z) \le M_1$ for all $z\in\bar{O}$ with bounded first order partial derivatives in $z$. 

\subsubsection*{Assumption 3}  There exist constants $\gamma , C_1 > 0$, which are
independent of $\eps$, such that $ \bE(\exp(\gamma \tau)|z^\eps_0 = z) \le C_1 < +\infty$.\\

\noindent
We define the generator of the dynamics $z_s^{\eps, u}$ by 
\begin{equation*}
\cL^{\eps}(u)\psi = \frac{1}{2\beta} a(z;\eps):\nabla^{2}\psi + \left(\sigma(z;\eps) u^\eps + b(z;\eps)\right)\cdot\nabla\psi \, .
\end{equation*}
Notice that the generator depends on the control $u$. When the control is absent we will use the notation $\cL^\eps := \cL^{\eps}(0)$.
The next result is standard (e.g., see \cite[Sec.~IV.2]{fleming2006})) and stated without proof. 

\begin{thm}\label{thm:verification}
Let $V^{\eps}\in C^{1,2}(O)\cap C(\bar{O})$ be the solution of the Hamilton-Jacobi-Bellman (HJB) equation
\begin{equation}\label{hjb}
\begin{aligned}
0 & = \min_{c\in\R^{n}}\left\{ \cL^{\eps}(c) V^{\eps} + L(z,c) \right\}, \\ 0
  & = V^\epsilon|_{\partial O}\,.
\end{aligned}
\end{equation}
Then
\[
V^{\eps}(z) = \min_{u} J^{\eps}(z;u^\eps)\,,
\]
where the minimum goes over all admissible feedback controls of the form
$u^{\eps}_{s}=c(z^{\eps,u}_{s}, s~;\eps)$. The minimizer is unique and is given by the feedback law 
\begin{equation} \label{eq:optControl} 
 \displaystyle \hat{u}^{\eps} = -\sigma(z;\eps)^{T}\nabla V^{\eps}(z) = \argmin_{c\in\R^{n}}\left\{ \cL^{\eps}(c) V^{\eps} + L(z,c) \right\}.
\end{equation} 
\label{thm-dp}
\end{thm}
The function $V^\epsilon$ is called \emph{value function} or \emph{optimal cost-to-go}. The homogenization problem for (\ref{sde})--(\ref{rgcost}) can be studied using a multiscale expansion of the nonlinear PDE (\ref{hjb}) in terms of the small parameter $\eps$; see, e.g., \cite{bensoussan1988,kushner1990}. In this article we remove the nonlinearity from the equation by means of a logarithmic transformation of the value function. Specifically, let
\[
\psi^{\eps}(z) = e^{-\beta V^{\eps}(z)}\,.
\]  
By chain rule,  
\begin{equation*}
\beta^{-1 }e^{\beta V^{\eps}}\cL^{\eps} e^{-\beta V^{\eps}}= -\cL^{\eps}V^{\eps} + \frac{1}{2}|\sigma^{T}\nabla V^{\eps}|^{2}\,,
\end{equation*}
which, together with the relation    
\begin{equation*}
-\frac{1}{2}|\sigma^{T}\nabla V^{\eps}|^{2} = \min_{c\in\R^{n}}\left\{\sigma
c\cdot\nabla V^\epsilon + \frac{1}{2}|c|^{2}\right\}\,,
\end{equation*}
implies that (\ref{hjb}) is equivalent to the linear boundary value problem 
\begin{equation} \label{linbvp}
\begin{aligned}
\left(\cL^{\eps} - \beta G\right)\psi^{\eps} \,  & = 0 \, ,\\ \psi^\epsilon|_{\partial O} & = 1\, ,
\end{aligned}
\end{equation}
for the function $\psi^{\eps}$. By the Feynman-Kac formula, \eqref{linbvp} has an interpretation as a 
control-free sampling problem (see \cite[Thm.~8.2.1]{oksendal2003}): 
\begin{equation} 
\psi^\epsilon(z) = \bE\left(\exp\left(-\beta \int_{0}^{\tau}
G(z^{\eps}_{s})\,ds\right)~ \bigg|~ z^{\eps}_{0}=z\right),
\label{phi-exp}
\end{equation}
where $z^{\eps}_{s}$ solves the control-free SDE
 \begin{equation*}
dz^{\eps}_{s} = b(z^{\eps}_{s};\eps)\,ds + \sigma(z^{\eps}_{s};\eps) \beta^{-1/2} dw_{s}\,.
\end{equation*}
Equations (\ref{hjb})--(\ref{phi-exp}) express a Legrendre-type duality between the value of 
an optimal control problem and cumulant generating functions \cite{DaiPra1996, fleming2006}: 
\begin{equation}
  V^{\eps} = -\beta^{-1}\log \psi^\epsilon \, .
  \label{dual}
\end{equation}
In other words, 
\begin{align*}
-\beta^{-1}\log\,
& \bE\left(\exp\left(-\beta\int_0^{\tau} G(z^\eps_s) ds\right) ~\bigg|~ z^\epsilon_0 = z\right)\\
&  = \inf\limits_{u_s} \bE\left(\int_0^{\tau} L(z^{\epsilon,u}_s, u_s) ds~\bigg| ~z^{\epsilon,u}_0 = z\right),
\end{align*}
where $z^{\eps,u}_s$ satisfies the controlled SDE (\ref{sde}) and
  $z^{\eps}_{s}=z^{\eps,0}_{s}$.\\

By the above assumptions and the strong
maximum principle for elliptic PDEs it follows that (\ref{linbvp}) has a
classical solution $\psi^{\eps}\in C^{1,2}(O)\cap C(\bar{O})$. Moreover,
combining Assumption 3, (\ref{phi-exp}) and H\"older's inequality, we have that 
\[
  \psi^\eps \ge \bE(\exp(-\beta M_1 \tau) | z^\eps_0 = z)
  \]
and 
\[
\bE(\exp(-\beta M_1 \tau) | z^\eps_0 = z)^{1/p}\bE(\exp(\gamma \tau) | z^\eps_0 = z)^{1/q} \ge 1
\]
where $p = \beta M_1/\gamma + 1$ and $q =\gamma/(\beta M_1) + 1$, and thus 
\[ 
  0< C_2 \le \psi^{\eps}\le 1, \quad\eps>0
\]  
for a constant $C_2 = C_1^{-\beta M_1 / \gamma}$ that is independent of $\eps$.
 
\begin{rem} In the course of the paper we will drop the assumption that the
  operator $\cL^{\eps}$ is uniformly elliptic and instead  require only that
  is hypoelliptic \cite{Malliavin1976}. In this case the matrix
  $\sigma\sigma^{T}$ can be semidefinite, if the vector field $b$ satisfies an
  additional controllability assumption, known as H\"ormander's condition
  \cite{Bismut1981}, which guarantees that the transition probability has a strictly positive  density with respect to Lebesgue measure, in which case (\ref{linbvp}) and (\ref{hjb}) have classical solutions; cf.~\cite[Sec.~IV]{fleming2006}.  
\end{rem}

%%%%%%%%%%%%%%%%%%%%%%%%%%%%%%%%%%%%%%%%%%%%
%% HOMOGENIZATION (uncontrolled case)
%%%%%%%%%%%%%%%%%%%%%%%%%%%%%%%%%%%%%%%%%%%% 

\subsection{Homogenization problem}\label{ssec:homo}

We now specify the class of multiscale systems considered in this article.  Specifically, we address   
slow-fast systems of the form 
\begin{subequations}\label{sdeXY}
\begin{eqnarray}
  &&  dx^{\eps}_{s} = \left( \frac{1}{\eps}f_0(x^{\eps},y^{\eps})+f_1(x^{\eps},y^{\eps})\right)ds + \beta^{-1/2}\alpha_1(x^{\eps},y^{\eps})  dw^1_{s}, \\
  &&dy^{\eps}_{s} = \left( \frac{1}{\eps^2}g_0(x^{\eps},y^{\eps})+\frac{1}{\eps}g_1(x^{\eps},y^{\eps})\right)ds + \frac{\beta^{-1/2}}{\eps}\alpha_2(x^{\eps},y^{\eps}) dw^2_{s}\,,\label{sdeadjoint}
\end{eqnarray}
\end{subequations}
together with an exponential expectation
\begin{equation}\label{phiXY}
\psi^\epsilon(x, y) = \bE\left(\exp\left(-\beta\int_0^{\tau} G(x^\epsilon_s,
y^\epsilon_s)~ ds\right) ~\bigg|~ x^\epsilon_0 = x, y^\epsilon_0 = y\right).
\end{equation}
Letting $\cL^{\eps}$ denote the infinitesimal generator of (\ref{sdeXY}), it holds that  
\begin{equation}
  \left(\cL^{\eps} - \beta G\right)\psi^\epsilon = 0,
\label{fk-homo}
\end{equation}
where 
\[
\cL^{\eps} = \frac{1}{\epsilon^2} \cL_0 + \frac{1}{\epsilon} \cL_1 + \cL_2\,,
\] 
with 
\begin{align*}
\cL_0 & = g_0 \cdot\nabla_y + \frac{1}{2} \beta^{-1} \alpha_2\alpha_2^T : \nabla_y^2 ,\\
\cL_1 & = f_0 \cdot\nabla_x + g_1\cdot \nabla_y , \\
\cL_2 & = f_1 \cdot\nabla_x + \frac{1}{2} \beta^{-1} \alpha_1\alpha_1^T :
  \nabla_x^2  \,. 
\end{align*}
Let us assume that $\psi^\epsilon$ admits the following perturbation expansion in powers of $\eps$:
\begin{align*}
\psi^\epsilon = \psi_0 + \epsilon\psi_1 + \epsilon^2 \psi_2 + \cdots\,.
\end{align*}
  By substituting the ansatz into (\ref{fk-homo}) 
and comparing different powers of $\epsilon$ we obtain a hierarchy of equations, the first three of which are  
\begin{equation}\label{terms-homo}
\begin{aligned}
\cL_0 \psi_0 & = 0, \\
\cL_0\psi_1 &=  - \cL_1\psi_0, \\
\cL_0\psi_2 & = - \cL_1\psi_1 - \cL_2\psi_0 + \beta G\psi_0\,.
\end{aligned}
\end{equation}
We suppose that for each fixed $x$, the dynamics (\ref{sdeadjoint}) of the fast variables are ergodic, with the unique invariant density $\rho_x(y)$. Then by construction $\rho_{x}$ is the unique solution of the  equation 
$\cL^*_0 \rho_x(y) = 0$, which together with the first equation of (\ref{terms-homo}) implies that $\psi_0$ is independent of $y$. 
In order to proceed, we further assume that $f_0(x,y)$ satisfies the \emph{centering condition}:
\begin{align*}
\int f_0(x,y)\rho_x(y) \,dy = 0\,.
\end{align*}
The centering conditions, together with the strong maximum principle implies that the solution of the cell problem
\begin{equation} \label{eq:cellEquation}
\cL_0\Theta(x,y) = -f_0(x, y)\,, \quad \int \Theta(x,y)\rho_x(y) \,dy = 0
\end{equation}
is unique, with $\psi_1(x,y) = \Theta(x,y)\cdot\nabla_x \psi_0(x)$. 
Multiplying $\rho_x(y)$ on both sides of the third equation in (\ref{terms-homo}) and
integrating with respect to $y$, we obtain
\begin{equation}
  \bar{\cL}\psi_0 - \beta \bar{G}\psi_0 = 0,
\label{fk2}
\end{equation}
where 
\begin{equation}\label{homo-generator1}
\bar{\cL}  = \bar{f}(x)\cdot\nabla_x + \frac{1}{2}\beta^{-1}
\bar{\alpha}\bar{\alpha}^T : \nabla_x^2 ,
\end{equation}
with  
\begin{align}
\begin{split}
\bar{f}(x)  = & \int \left[\nabla_x\Theta(x,y)f_0(x,y) +
\nabla_y\Theta(x,y)g_1(x,y) + f_1(x, y)\right] \rho_x(y) \,dy , \\
\bar{G}(x)  = & \int G(x, y) \rho_x(y) \,dy , \\
\bar{\alpha}(x)\bar{\alpha}(x)^T  = & \int \left[\beta\left(\Theta(x,y)f_0(x,y)^T +
f_0(x,y)\Theta(x,y)^T\right) \right. \\
& + \left. \alpha_1(x,y)\alpha_1(x,y)^T\right] \rho_x(y)\, dy \, .
\end{split}
\label{homo-generator2}
\end{align}

\subsection*{Homogenized control system}
It follows using standard homogenization theory for linear elliptic equations (e.g. \cite{bensoussan1978,pavliotis2008}) that   
for $\epsilon \rightarrow 0$ the solution of (\ref{fk-homo}) converges to the leading term of the asymptotic expansion:  
\begin{equation}
\psi_0(x) = \bE\left(\exp\left(-\beta\int_0^\tau \bar{G}(x_s)\, ds\right) ~\bigg|~ x_0 = x\right) , \label{phi1} 
\end{equation}
where $x_s$ is the solution of the homogenized SDE
\begin{equation}
  d x_s = \bar{f}(x_s) ds + \bar{\alpha}(x_s) \beta^{-1/2} dw_s\, ,
\label{reduced-eqn-2}
\end{equation}
with coefficients as given in (\ref{homo-generator2}).

The corresponding asymptotic expansion of the value function $V^\epsilon$  for $\epsilon\rightarrow 0$ is obained by the logarithmic transformation~\eqref{dual}: 

\begin{align*}
  V^\epsilon =-\beta^{-1}\log (\psi_0 + \epsilon\psi_1 + \cO(\epsilon^{2})) =-\beta^{-1}\log \psi_0 - \beta^{-1}\frac{\psi_1}{\psi_0}\epsilon + \cO(\epsilon^{2}).
\end{align*}
Therefore, using the ansatz $V^\epsilon = V_0 + \epsilon V_1 + \epsilon^2 V_2
+\cdots$ it follows that 
\[
  V_0 = -\beta^{-1}\log \psi_0, \quad V_1 = -\beta^{-1}\frac{\psi_1}{\psi_0}.
\]
Using the log-transformation property of the cumulant generating function (p.~\pageref{dual}), we conclude that $V_0$ is
the value function of the optimal control problem
\begin{align*} 
V_{0}(x) = \inf_{u} \bE\left(\int_0^{\tau} \left[\bar{G}(x^{u}_s) + \frac{1}{2} |u_s|^2\right] ds ~\bigg|~x^u_0 = x \right),
\end{align*}
where the minimization is subject to the homogenized dynamics 
\begin{align}   \label{reduced-eqn-2-controlled}
  d x^u_s = (\bar{f}(x^u_s) + \bar{\alpha}(x^u_s)u_s) ds +
  \bar{\alpha}(x^u_s) \beta^{-1/2} dw_s\,.
\end{align}
According to \eqref{eq:optControl}, the optimal feedback law for the homogenized problem reads 
\begin{equation} \label{eq:optContHom}
  \hat{u}_{t}=-\bar{\alpha}(x_{t}^{u})^T\nabla V_0(x^{u}_{t}).
\end{equation}

\subsection{Control of the full dynamics using reduced models}
\label{sec:reducedModels}
Our goal is to find the optimal control policy $\hat{u}^\epsilon =
(\hat{u}^{1,\epsilon}, \hat{u}^{2,\epsilon})$ for the fast/slow system~\eqref{sdeXY} for $\eps \ll 1$. Using Theorem~\ref{thm-dp} and the asymptotic expansion of $V^\epsilon$, we have
\begin{align}
  \begin{split}
    &\hat{u}^{1,\epsilon} = -\alpha_1^T\nabla_x V^\epsilon = -\alpha_1^T\nabla_x V_0 +
\cO(\epsilon), \\
&\hat{u}^{2,\epsilon} =
-\frac{1}{\epsilon} \alpha_2^T\nabla_y V^\epsilon
= -\alpha_2^T\nabla_y V_1 + \cO(\epsilon) = -\alpha_2^T\nabla_y \Theta \nabla_x
V_0 + \cO(\epsilon).
\end{split}
\label{opt-control-1}
\end{align}

Notice that the leading terms in (\ref{opt-control-1}) 
are related to the value function of optimal control
problem for the reduced SDE. This indicates that we may design the control
policy from the reduced problem and use it to control the original
multiscale equation. This assertion is justified by the following result for
the general optimal control problem \eqref{sde}--\eqref{rgcost}.

\begin{comment}
 Let us substitute the asymptotic form of the control $u^\eps$ \eqref{opt-control-1} in the original multiscale SDE's \eqref{sdeXY}, 

\begin{align} \label{sdeXYcont}
\begin{split}
dx^{\eps, u}_{t} =& \left( \frac{1}{\eps}f_0(x^{\eps,u},y^{\eps,u})+f_1(x^{\eps,u},y^{\eps,u}) -\alpha_1\alpha_1^T \nabla V_0\right)dt + \alpha_1(x^{\eps,u},y^{\eps,u}) \beta^{-1/2} dw^1_{t}, \\
dy^{\eps, u}_{t} =& \left( \frac{1}{\eps^2}g_0(x^{\eps,u},y^{\eps,u})+\frac{1}{\eps}(g_1(x^{\eps,u},y^{\eps,u})-\alpha_2\alpha_2^T\nabla_y\Theta \nabla V_0) + u_2^2(\eps) \right)dt \\
& + \frac{\beta^{-1/2}}{\eps}\alpha_2(x^{\eps,u},y^{\eps,u}) dw^2_{t},
\end{split}
\end{align}
where the specific form of $u_2^2$ is not important except that is $\cO(\eps^{2})$.
We can apply the homogenization formulas \eqref{homo-generator2} in the
previous equations and, after some algebra, arrive precisely at the reduced
controlled equation \eqref{reduced-eqn-2-controlled}. 
\end{comment}
\begin{thm}
Let Assumptions 1,2 and 3 hold and, furthermore, suppose that $\eps < (\gamma/\beta)^{1/2}$ and
  $|u_{t}-\hat{u}_{t}| \le \epsilon$ uniformly in $t$. Then we have
\begin{equation}
  \left|J(u) - J(\hat{u}^\eps)\right| \le C \eps^2.
\end{equation}
\label{thm:entropyIndefinite}
\end{thm}

The proof of the theorem can be found in Appendix~\ref{sec-cost_function}. \\

Upon combining the above theorem with the formula for the optimal control policy in~\eqref{opt-control-1} we conclude that when the two time scales in the system are well separated, $\epsilon \ll 1$, the optimal control policy is well approximated by the leading order terms in~\eqref{opt-control-1} and results in a cost value that is nearly optimal.

\begin{rem}
All considerations in this paper readily generalize to the averaging problem,
i.e. when $f_{0}=g_{1}=0$ in (\ref{sdeXY}). This is not surprising since for
averaging problems strong convergence $\psi^{\eps}\to\psi$ is expected to hold (when the diffusion
coefficient $\alpha_1$ in~\eqref{sdeXY} is independent of the fast variable
$y$). Related problems have been addressed in \cite{PapPavSt08}, in which the authors study parameter estimation 
and convergence of the maximum likelihood function under averaging and homogenization.
\end{rem}

%%%%%%%%%%%%%%%%%%%%%%%%%%%%%%%%%%%%%%%%%%%%
%% 3 EXAMPLES
%%%%%%%%%%%%%%%%%%%%%%%%%%%%%%%%%%%%%%%%%%%%

\section{Three prototypical applications}
\label{sec-examples}
In this section we apply the results presented in the previous section to three typical
multiscale models. For each model we first state the optimal
control problem along with its log-transformed counterpart, then we study the asymptotic
limits of the value function and of the optimal control policy and give explicit formulae for the solution. 
The first two examples are taken from \cite{PapPavSt08}, while the third is adapted from~\cite{gugercin2004}.

%%%%%%%%%%%%%%%%%%%%%%%%%%%%%%%%%%%%%%%%%%%%
\subsection{Overdamped Langevin equation}
\label{subsec-examples-langevin}
We consider the second-order Langevin equation
\begin{equation}
\epsilon^2 \frac{d^2 x^\epsilon}{ds^2} = - \frac{dx^\epsilon}{ds} -\nabla \Phi(x^\epsilon) +
\sqrt{2} \beta^{-1/2} \frac{dw}{ds}, \label{langevin}
\end{equation}
where $\epsilon \ll 1,\, x \in \mathbb{R}^n,\, \beta > 0$, and $\Phi$ being a smooth the potential energy function. Introducing the
auxiliary variable $y^\epsilon$ we can recast (\ref{langevin}) as 
\begin{subequations}
\begin{eqnarray}
&&\frac{dx^{\epsilon}}{ds} = \frac{1}{\epsilon} y^{\epsilon}, \\
&&d y^{\epsilon} = -\left(\frac{1}{\epsilon}\nabla \Phi(x^{\epsilon}) + \frac{1}{\epsilon^2}y^{\epsilon}\right)dt 
+  \frac{1}{\epsilon} \sqrt{2} \beta^{-1/2} dw\,.
\end{eqnarray}
\label{langevin-2var}
\end{subequations}
We consider the solution of the optimal control problem 
\begin{align} \label{value-fun-indefinite}
V^\epsilon(x, y) &= \inf_{u^\eps} \bE\left(\int_0^{\tau} \left[G(x^{\epsilon, u}_s) + \frac{1}{2} |u^\eps_s|^2\right] ds~\bigg|~x^{\epsilon, u}_0 = x, y^{\epsilon, u}_0 = y\right)
\end{align}
under the controlled Langevin dynamics 
\begin{subequations}
\begin{eqnarray}
&&\frac{dx^{\epsilon,u}_s}{ds} = \frac{1}{\epsilon} y^{\epsilon,u}_s, \\
&&d y^{\epsilon,u}_s = \left( \frac{1}{\epsilon} \sqrt{2}u^\eps_s - \frac{1}{\epsilon}\nabla \Phi(x^{\epsilon,u}_s)
- \frac{1}{\epsilon^2}y^{\epsilon,u}_s\right)dt  +
 \frac{1}{\epsilon} \sqrt{2}  \beta^{-1/2} dw.
\end{eqnarray}
\label{ex1-control-eqn}
\end{subequations}
We notice that~\eqref{langevin-2var} is somewhat
different to the form specified in Section \ref{sec:mms}, since there is no noise and hence no control
term in the equation for $x^\epsilon$. The infinitesimal generator correpsonding to (\ref{langevin-2var}) is 
hypoelliptic (rather than elliptic). Yet the standard homogenization arguments apply, for here the fast variable
is $y$ and the noise is acting uniformly in $y$. As a consequence the generator of the fast dynamics is uniformly 
elliptic, ans hence the standard theory applies. 
Let
  \[
  \psi^\epsilon(x,y) = \bE\left(\exp\left(-\beta\int_0^{\tau} G(x^\epsilon_s) ~ds\right) ~\bigg|~ x^\epsilon_0 = x, y^\epsilon_0 = y\right).
  \]
Assuming that the linear boundary value problem (\ref{linbvp}) associated with $\psi^{\eps}$ has a classical solution, then the dual relation 
$V^\epsilon = -\beta^{-1}\log \psi^\epsilon$ holds and the results of the
previous section carries over without alternations.

\subsection*{Homogenized control system}

From the above and the considerations from the previous section we can conclude that the leading term of $V^\epsilon(x,y)$ 
satisfies the optimal control problem of the homogenized SDE, which is 
\begin{equation}\label{v0-langevin}
V_{0}(x) = \inf_{u} \bE\left(\int_0^{\tau} \left[G(x^u_s) + \frac{1}{2}
|u_s|^2\right] ds ~\bigg|~x^u_0 = x\right)
\end{equation}
subject to the homogenized equation
\begin{equation}\label{e:overdamped}
dx^u_s = -\nabla \Phi(x^u_s) ds+ \sqrt{2} u_s ds +\sqrt{2} \beta^{-1/2} dw_s.
\end{equation}
Equation (\ref{e:overdamped}) is called the \emph{overdamped Langevin equation} that is obtained from (\ref{langevin}) 
by letting the inertial second-order term tend to zero \cite{nelson1967}.

We now derive an explicit asymptotic expression for the optimal feedback law $\hat{u}_{t}^{\eps}:=\hat{u}^{2,\eps}_{t}$, with $\hat{u}_{t}^{\eps} = \hat{c}^{\eps}(x^{\eps,u}_{t},y^{\eps,u}_{t})$ and 
\[
\hat{c}^\epsilon = -\sqrt{2}\epsilon^{-1} \nabla_y V^\epsilon(x,y)\,.
\]
From (\ref{ex1-control-eqn}) and the
expansion $\psi^\epsilon(x,y) = \psi_0(x) + \epsilon \psi_1(x,y) +
\cO(\eps)\epsilon)$ we find 
\begin{equation}\label{ex1-u-1}
\begin{aligned}
  \hat{c}^\epsilon  & = -\sqrt{2}\nabla_y V_1 + \cO(\eps)
  =-\sqrt{2}\nabla_y\Theta\nabla_x V_0 + \cO(\eps).
\end{aligned}
\end{equation}
As before $\Theta$ is the solution to the associated cell problem. To solve it we notice that the infinitesimal generator 
of (\ref{langevin-2var}) has the form
\begin{equation*}
\cL = \frac{1}{\epsilon^2}\cL_0 + \frac{1}{\epsilon}\cL_1
\end{equation*}
with 
\begin{align}
\cL_{0} & = -y\cdot \nabla_y + \beta^{-1} \Delta_y\\
\cL_{1} & = y\cdot \nabla_x - \nabla \Phi\cdot \nabla_y\,,
\end{align}
which implies that the cell problem for $\Theta$ reads  
\[
\cL_0\Theta  = -y\, ,
\] 
with unique solution $\Theta(x,y) = y$. Combining it with (\ref{ex1-u-1}), we obtain the sought asymptotic expression for the optimal feedback law: 
\begin{equation}
  \hat{c}^\epsilon =
-\sqrt{2}\nabla _xV_0 + \cO(\eps)\,,
\label{ex1-u}
\end{equation}
with $V_{0}$ as given in (\ref{v0-langevin}). We therefore conclude that the optimal control 
$\hat{u}^\epsilon$ for the Langevin equation \eqref{langevin} converges to the optimal control 
of the overdamped equation~\eqref{e:overdamped} as $\epsilon \rightarrow 0$. Moreover,  
Theorem~\ref{thm:entropyIndefinite} guarantees that the control value is asymptotically exact if we 
replace $\hat{u}^\epsilon$ with the control $\hat{u}=-\sqrt{2}\nabla_x V_0$ in
the multiscale dynamics (\ref{ex1-control-eqn}). Hence the overdamped equation is backward stable.

\subsection*{Langevin dynamics in a double-well potential}

As an example consider the case $n=1$, with running cost $G(x) = 1$ in (\ref{value-fun-indefinite}) and   
random stopping time   
\[
\tau= \inf\{s > 0 : x^{\epsilon,u}_s > 2\}\,.
\] 
The dynamics are governed by the double-well potential 
\[
\Phi(x) = \frac{1}{4}(x^2-1)^2
\]
depicted in Figure \ref{fig-ex1A}. As the homogenized problem is one-dimensional, 
the leading term $V_0$ of the value function $V^\epsilon$ can be computed by solving a 
two-point boundary value problem. The resulting leading term (\ref{ex1-u}) for the optimal control 
\[\hat{u}^\epsilon_{t}=\hat{c}^{\eps}(x^{\eps,u}_{t})
\] 
is shown in Figure \ref{fig-ex1C}. We then computed
the cost function $J^{\eps} = J(\hat{u}^{\eps})$ starting from
three different initial points $x_0 = 1.0, 1.2, 1.5$, using the approximation 
\[
\hat{u}^\epsilon_{t}\approx -\sqrt{2}\nabla _xV_0(x^{\eps,u}_{t})\,.
\] 
Figure \ref{fig-ex1D} clearly shows that $J^\epsilon$ approaches its infimum $V_0(x_0)$ as
$\epsilon\rightarrow 0$. A clear advantage of controlling the full dynamics using the optimal
control obtained from the reduced model here is that the infinitesimal generator $\cL^{\eps}$ 
of the original Langevin dynamics is not self-adjoint, whereas the infinitesimal
generator $\bar{\cL}$ of the reduced dynamics is essentially self-adjoint. 
That is, not only do we benefit from a lower dimensionality of the reduced-order model (by a factor of 2), 
but we also avoid solving a boundary value problem with a non-selfadjoint operator.

\begin{figure}
\centering
\begin{subfloat}[]{
\includegraphics[width=5cm]{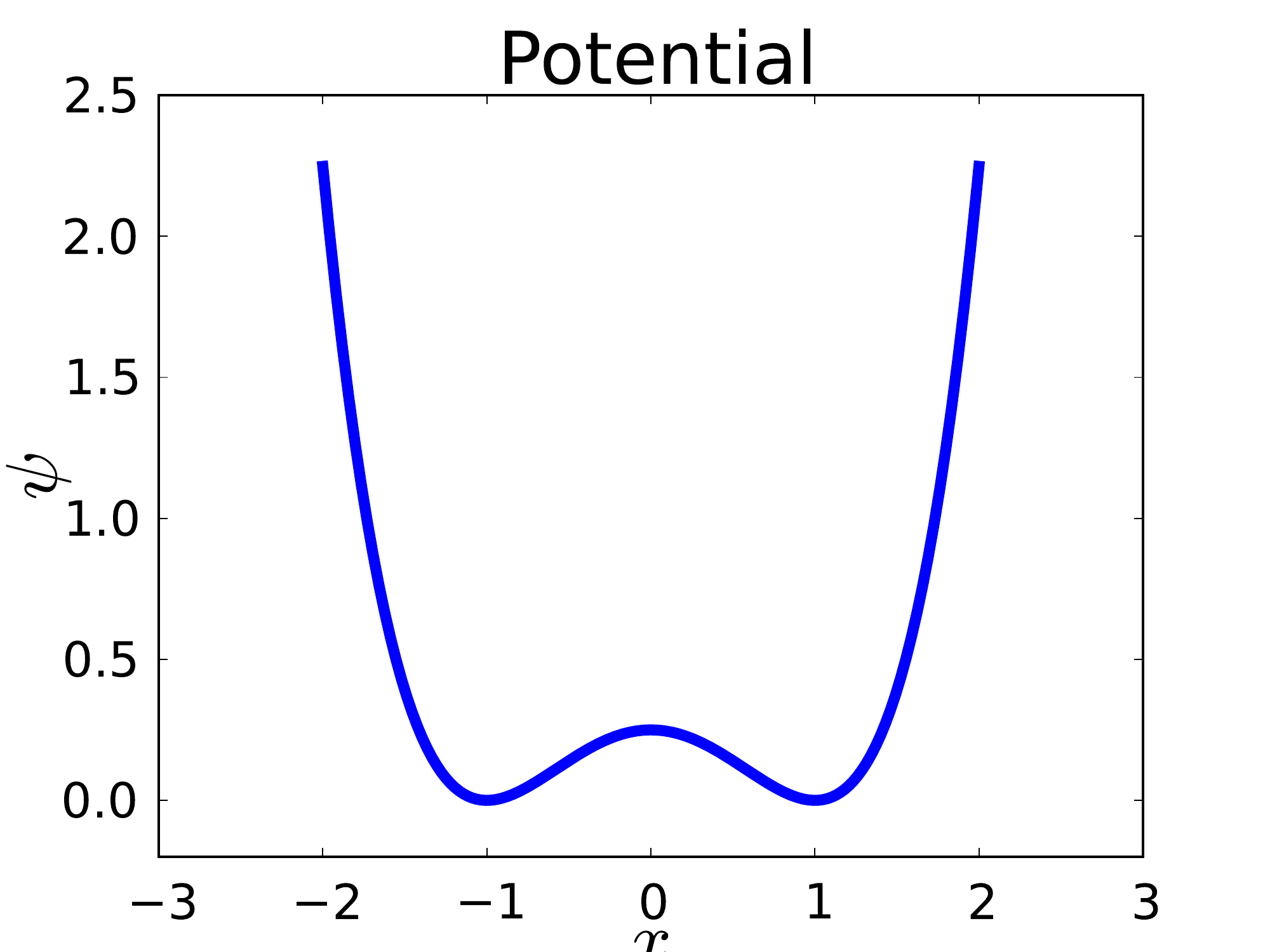} 
\label{fig-ex1A}
}
\end{subfloat}
\begin{subfloat}[]{
\includegraphics[width=5cm]{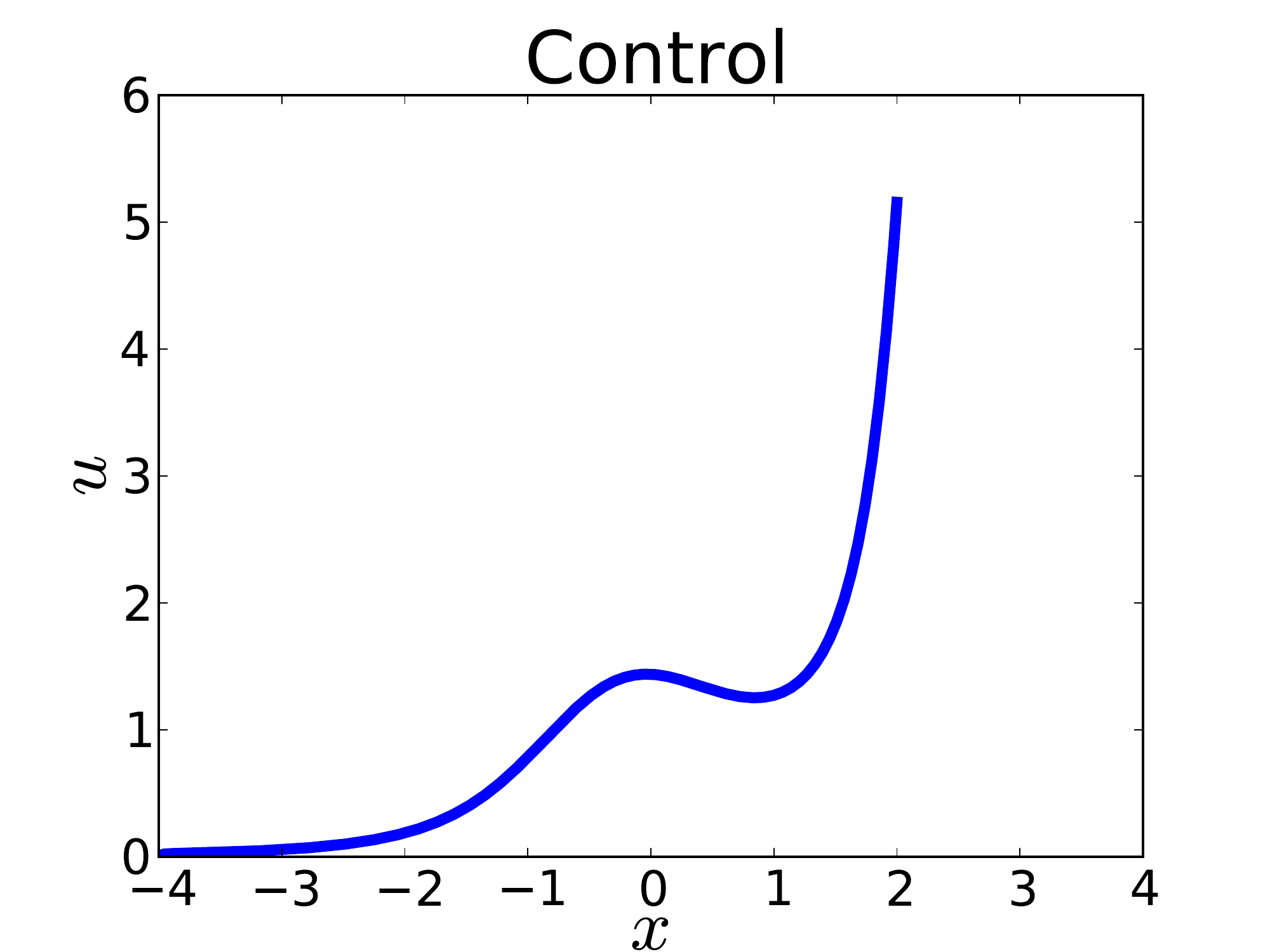}
\label{fig-ex1C}
}
\end{subfloat}
\caption{Overdamped Langevin equation. (A) Double-well potential $\Phi(x)$.  (B) Leading term of optimal control
in (\ref{ex1-u}). \label{fig-ex1}}
\end{figure}
\begin{figure}
\centering
\includegraphics[width=9cm]{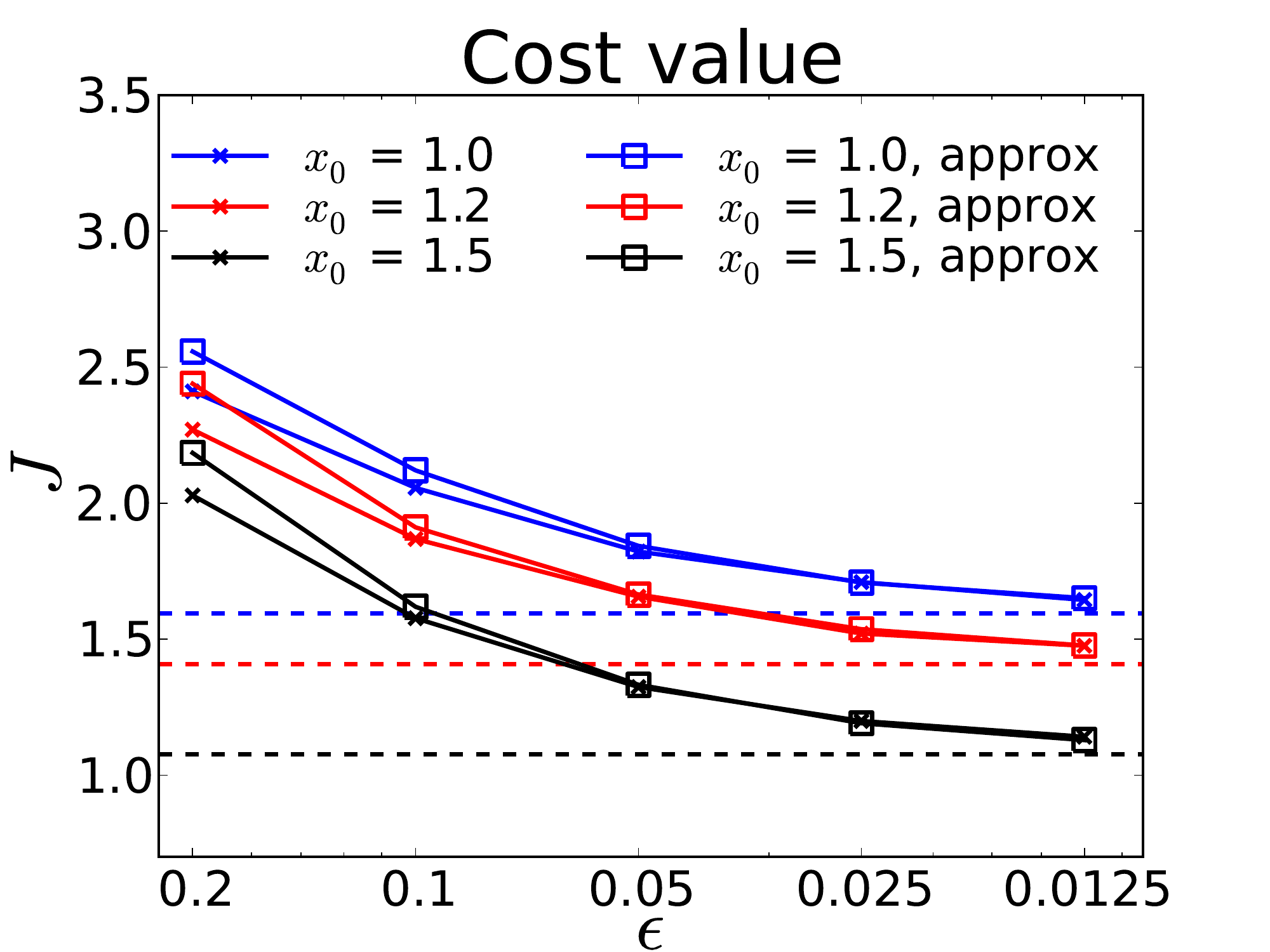}
\caption{Overdamped Langevin dynamics. Cost function for different values of $\epsilon$. Different colors correspond to 
different initial values $x_0$. Lines marked with ``$\times$'' are the value function $V^\epsilon$ computed
from the exponential expectation using Monte-Carlo. Lines marked with
``$\square$'' are the cost function $J^\epsilon = J(\hat{u}^{\eps})$, computed from the homogenized 
control with the original dynamics. We observe that the two values approach $V_0(x_0)$ as $\epsilon \rightarrow 0$ (horizontal
line). \label{fig-ex1D}}
\end{figure}

%%%%%%%%%%%%%%%%%%%%%%%%%%%%%%%%%%%%%%%%%%%%
%%%%%%%%%%%%%%%%%%%%%%%%%%%%%%%%%%%%%%%%%%%%
\subsection{Diffusion in a periodic potential}
\label{subsec-examples-homo}
We now consider the SDE \cite{is_multiscale, pavliotis2008}
\begin{align}
dx^\epsilon_s &= -\nabla \Phi^\epsilon(x^\epsilon_s)ds + \sqrt{2}\beta^{-1/2} dw_s  
\label{ex2-dynamics}
\end{align}
where $\beta > 0$ and $\Phi^\epsilon(x) = \Phi_{0}(x) + p(x/\epsilon)$, with $p(y)$ being a smooth, $1$-periodic
function (see Fig.~\ref{fig:msPot} below). 
We study the optimal control problem
\begin{align} \label{eq:valueFuncOneVar}
V^\epsilon(x) &= \inf_{u^{\eps}} \bE\left(\int_{0}^{\tau} G(x^{\epsilon, u}_s) + \frac{1}{2} |u^\eps_s|^2 ds~\bigg|~x^{\epsilon, u}_0 = x \right),
\end{align}
where 
\begin{align}
dx^{\epsilon,u}_s &= -\nabla \Phi^\epsilon(x^{\epsilon,u}_s)ds +
\sqrt{2}u^\eps_sds + \sqrt{2} \beta^{-1/2} dw_s
\label{controled_ex2}
\end{align}  
and $\tau=\tau^{\eps,u}$ is the first hitting time of the set $\{x\ge 1.5\}$ (blue region in Fig.~\ref{fig:msPot}).
\begin{figure}
\centering
  \includegraphics[width=80mm]{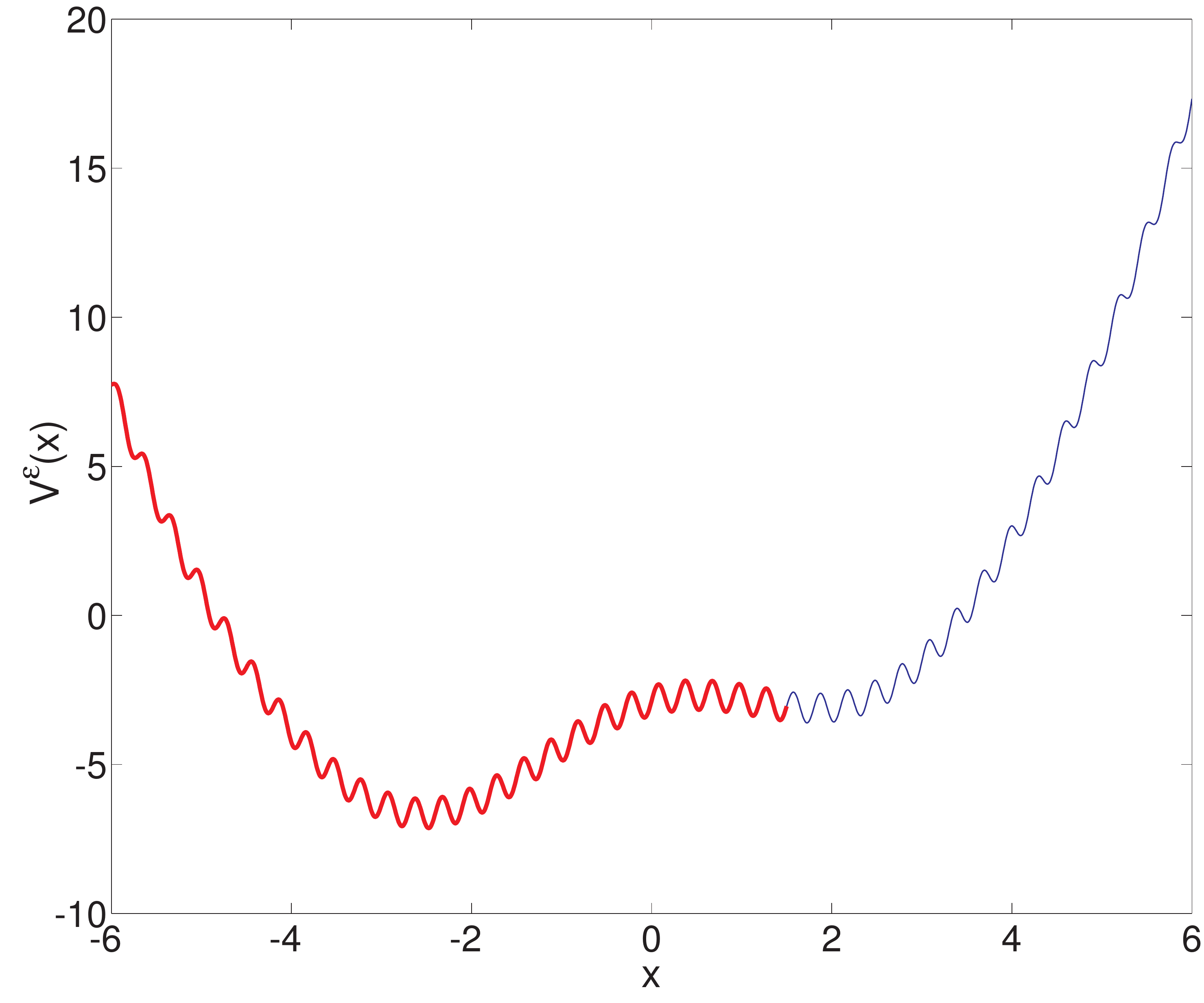}
 \caption{Controlled diffusion in a multiscale potential: minimize the transition time from the red to the blue region. \label{fig:msPot}}
 \end{figure}

In order to relate this system with the homogenization problem studied in Section~\ref{ssec:homo}, we introduce the auxiliary variable
$y^{\epsilon} = x^{\epsilon}/{\epsilon}$ and reformulate
(\ref{ex2-dynamics}) as
\begin{subequations}
\begin{eqnarray}
  &&dx^{\epsilon}_s = -\frac{1}{\epsilon}\nabla p(y^{\epsilon}_s) ds -\nabla \Phi_{0}(x^{\epsilon}_s)ds + \sqrt{2} \beta^{-1/2} dw_s,  \\
  &&dy^{\epsilon}_s = -\frac{1}{\epsilon^2}\nabla p(y^{\epsilon}_s) ds -\frac{1}{\epsilon}\nabla \Phi_{0}(x^{\epsilon}_s)ds +
\frac{1}{\epsilon} \sqrt{2} \beta^{-1/2} dw_s,
\end{eqnarray}
\label{ex2-sde-1}
\end{subequations} 
where $x^{\epsilon}_s,  y^{\epsilon}_s$ are driven by the same noise $w_s$. The associated value function reads
\begin{align*}
\widetilde{V}^\epsilon(x, y) &= \inf_{u^{\eps}} \bE\left(\int_{0}^{\tau} G(x^{\epsilon, u}_s) + \frac{1}{2}|u^\eps_s|^2 ds~\bigg|~x^{\epsilon, u}_0 = x, y^{\epsilon, u}_0 =
y \right),
\end{align*}
with 
\begin{subequations} \label{eq:controledPeriodicMS}
\begin{align} 
  &  dx^{\epsilon,u}_s = -\frac{1}{\epsilon}\nabla p(y^{\epsilon,u}_s) ds -\nabla	
\Phi_{0}(x^{\epsilon,u}_s)ds + \sqrt{2} u^\epsilon_s\, ds + \sqrt{2} \beta^{-1/2} dw_s,  \\
&dy^{\epsilon,u}_s = -\frac{1}{\epsilon^2}\nabla p(y^{\epsilon,u}_s) ds
-\frac{1}{\epsilon}\nabla \Phi_{0}(x^{\epsilon,u}_s)ds +
\frac{1}{\epsilon} \sqrt{2} u^\epsilon_s\, ds +
\frac{1}{\epsilon} \sqrt{2} \beta^{-1/2} dw_s \, .
\end{align}
\end{subequations}
Notice that the same noise and the same control are applied to both equations.
Clearly $V^\epsilon(x) =
\widetilde{V}^\epsilon(x,x/\epsilon)$ and the dual relation
 $\widetilde{V}^\epsilon(x,y) = -\beta^{-1}\log \psi^\epsilon(x, y)$ applies, where
 $\psi^\eps$ is defined as in Section \ref{ssec:homo}. The generator of
(\ref{ex2-sde-1}) now is 
\[
\cL^\eps = \frac{1}{\epsilon^2}\cL_0 + \frac{1}{\epsilon}\cL_1+\cL_2 ,
\]
with 
\begin{align*}
\cL_{0} &= -\nabla p\cdot\nabla_y + \beta^{-1}\Delta_y , \\
\cL_{1} & = -\nabla p\cdot\nabla_x-\nabla \Phi_{0}\cdot
\nabla_y+2\beta^{-1}\nabla_x\nabla_y , \\
\cL_{2} & = -\nabla \Phi_{0} \cdot
\nabla_x+\beta^{-1}\Delta_x\,.
\end{align*}

\subsection*{Homogenized control system}
Applying the results of Section \ref{sec:mms}, we conclude that the
leading term of $V^\epsilon(x)$ is the value function of the following reduced-order optimal control problem: minimize 
\begin{equation} 
J(u)= \bE\left(\int_{0}^{\tau}G(x^u_s) + \frac{1}{2} |u_s|^2 ds\right),
\label{eq:valueFuncHom}
\end{equation}  
subject to the homogenized dynamics
\begin{equation} \label{eq:controlledSDEhomo}
dx^u_s = -K\nabla \Phi(x^u_s)ds + \sqrt{2K} u_s ds  + \sqrt{2K} \beta^{-1/2} dw_s,
\end{equation}
with the effective diffusivity
\begin{equation*}
K=\int (I + \nabla_y \Theta(y))(I + \nabla_y \Theta(y))^T \rho(y) \,dy\,.
\end{equation*}
In the above formula $\rho(y) = Z^{-1}\exp(-\beta p(y))$ denotes the invariant density 
of the fast variable $y$ and $\Theta(y)$ is the solution of the Poisson equation
\begin{displaymath}
\cL_0\Theta(y)=\nabla p(y).
\end{displaymath}
Specifically, we have (cf.~\cite{PavlSt06} for details)
\begin{displaymath}
K^{-1} =  \int_0^1 \exp(-\beta p(y))\, dy  \int_0^1 \exp(\beta p(y))\, dy\,. 
\end{displaymath}
 
The value function of the homogenized control problem \eqref{eq:valueFuncHom}--\eqref{eq:controlledSDEhomo} and the corresponding optimal control satisfy 
\[
V_0(x)=-\beta^{-1} \log \psi_0(x)
\]
and 
\begin{equation} \label{eq:optHomogenizedControl}
\hat{u}_{t}=-\sqrt{2K} \nabla V_0(x^{\hat{u}}_{t}), 
\end{equation}
where
\begin{equation*}
\bar{\cL} \psi_0(x) = K \cL_2 \psi_0(x) = \beta G(x) \psi_0(x), \quad 
\psi_0(x)\big|_{ \partial O}=0,
\end{equation*}
as given in \eqref{fk2}.

\subsection*{Reduced model is not backward stable} 
In contrast to the previous example, however, the optimal control $\hat{u}$ obtained 
from the homogenized equation alone does meet the requirements of backward stability. This can be understood by 
noting that the optimal control
the original dynamics is given by  the feedback law
\begin{equation}\label{ex2-u}
\begin{aligned}
  c^\epsilon(x) &= -\sqrt{2}\nabla V^\epsilon(x) =
\sqrt{2}\beta^{-1}\frac{\nabla_x\psi^\epsilon(x,x/\epsilon)}{\psi^\epsilon(x,x/\epsilon)}\\
& = 
\sqrt{2}\beta^{-1}\frac{\nabla_x\psi_0(x)+\nabla_y\psi_1(x,x/\epsilon)}{\psi_0(x)}
+ \cO(\eps),
\end{aligned}
\end{equation}
which can be formally derived from the expansion
\[
\psi^\epsilon(x,x/\epsilon) = \psi_0(x) +
\epsilon\psi_1(x, x/\epsilon) + \ldots\,.
\]  
After some manipulations we find that the asymptotic expression for $c^\epsilon$ reads 
\begin{equation}\label{ex2-u-2}
\begin{aligned}
 c^\epsilon(x) =& \sqrt{2}\beta^{-1} \frac{\exp(\beta p(x/\epsilon)) }{\int_0^1
\exp(\beta p(z))\,dz }  \frac{\psi_0'(x)}{\psi_0(x)} +
\cO(\epsilon) \\
&= \frac{\exp(\beta p(x/\epsilon)) }{\sqrt{K}\int_0^1
\exp(\beta p(z))\,dz } c(x) + \cO(\epsilon),
\end{aligned}
\end{equation}
where we used the shorthand $c(x)=-\sqrt{2K} \nabla V_0(x)$ in the last row. Therefore we conclude that $c^{\eps}$ must be of the form
\[
c^{\eps}(x) = \tilde{c}(x,x/\eps)+ \cO(\epsilon)\,.
\]
Yet $\tilde{c}(x, x/\epsilon)$ does not converge to $c(x)$ in any reasonable norm, for the $x/\eps$ part keeps oscillating as $\eps\to 0$. What does converge, however, is the  \emph{average}: 
\begin{displaymath}
  \int_0^1 \tilde{c}(x,y)\rho(y) d y =  \int_0^1 \tilde{c}(x,y)
  \frac{e^{-\beta p(y)}}{\int_0^1 e^{-\beta p(z)} d z} d y = \sqrt{K} c(x)\,.
\end{displaymath}
This fact is illustrated in Figure \ref{fig:multiscaleControl} which shows the oscillations of order one that are a consequence of the $\eps$-periodic oscillations of the value function; since the optimal control law involves the derivative of the value function, oscillations of size $\eps$ in the value function turn into $\cO(1)$ contributions to the optimal control. Figure \ref{fig:convVF} shows the difference between the homogenized value function $V_0(x)$ and its multiscale counterpart $V^\epsilon(x)$ in the $L^{2}$-norm. The figure also shows the $L^{2}$-difference between the multiscale optimal feedback law $c^\epsilon(x)$ and the corrected homogenized feedback law $\tilde{c}(x,x/\epsilon)$, including the oscillatory correction. This demonstrates strong $\cO(\eps)$ convergence  in $L^{2}$ of both value function and optimal control.   
 
\begin{figure}
\centering
   \includegraphics[width=80mm]{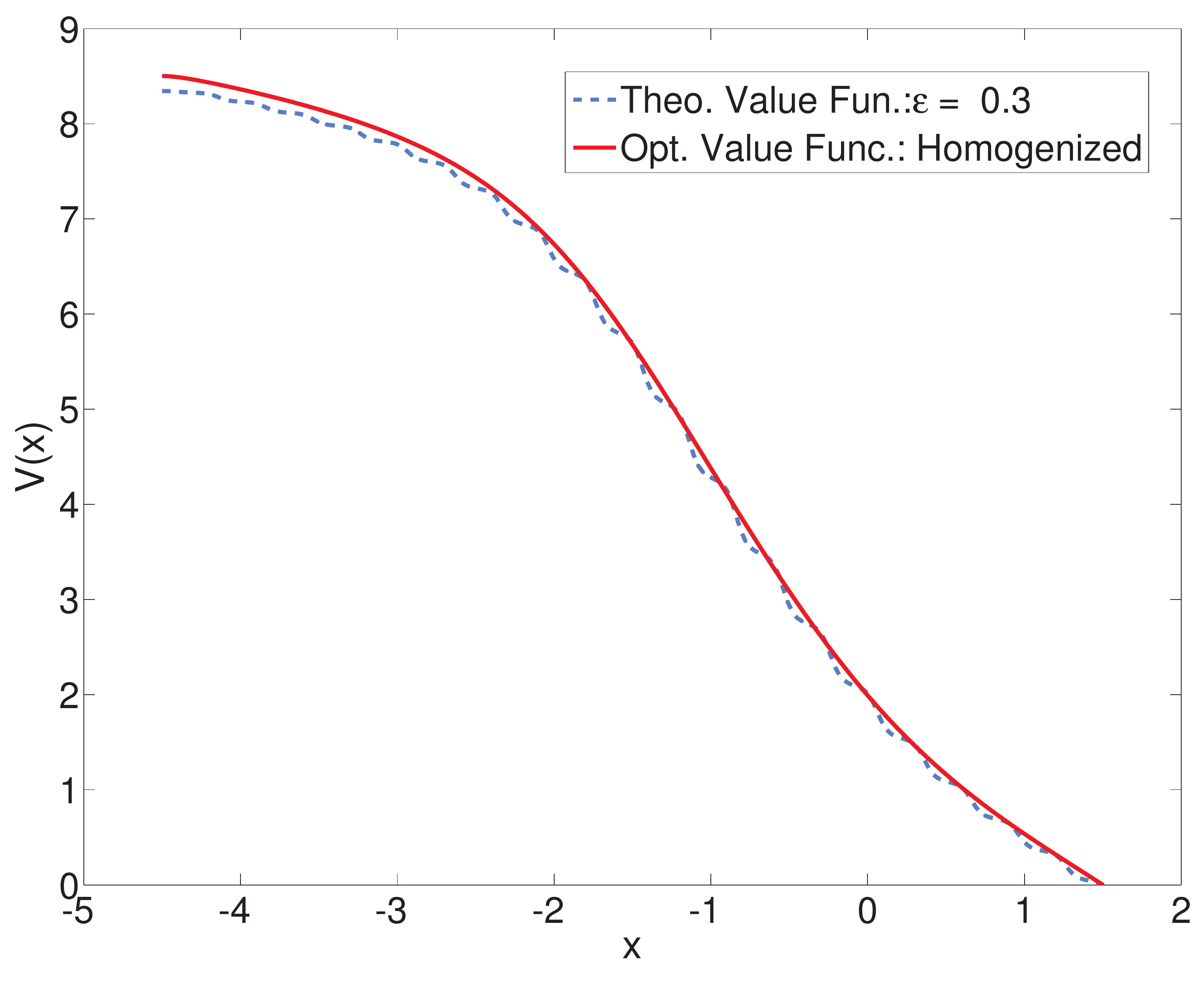}\\
\hspace{10pt}  \includegraphics[width=80mm]{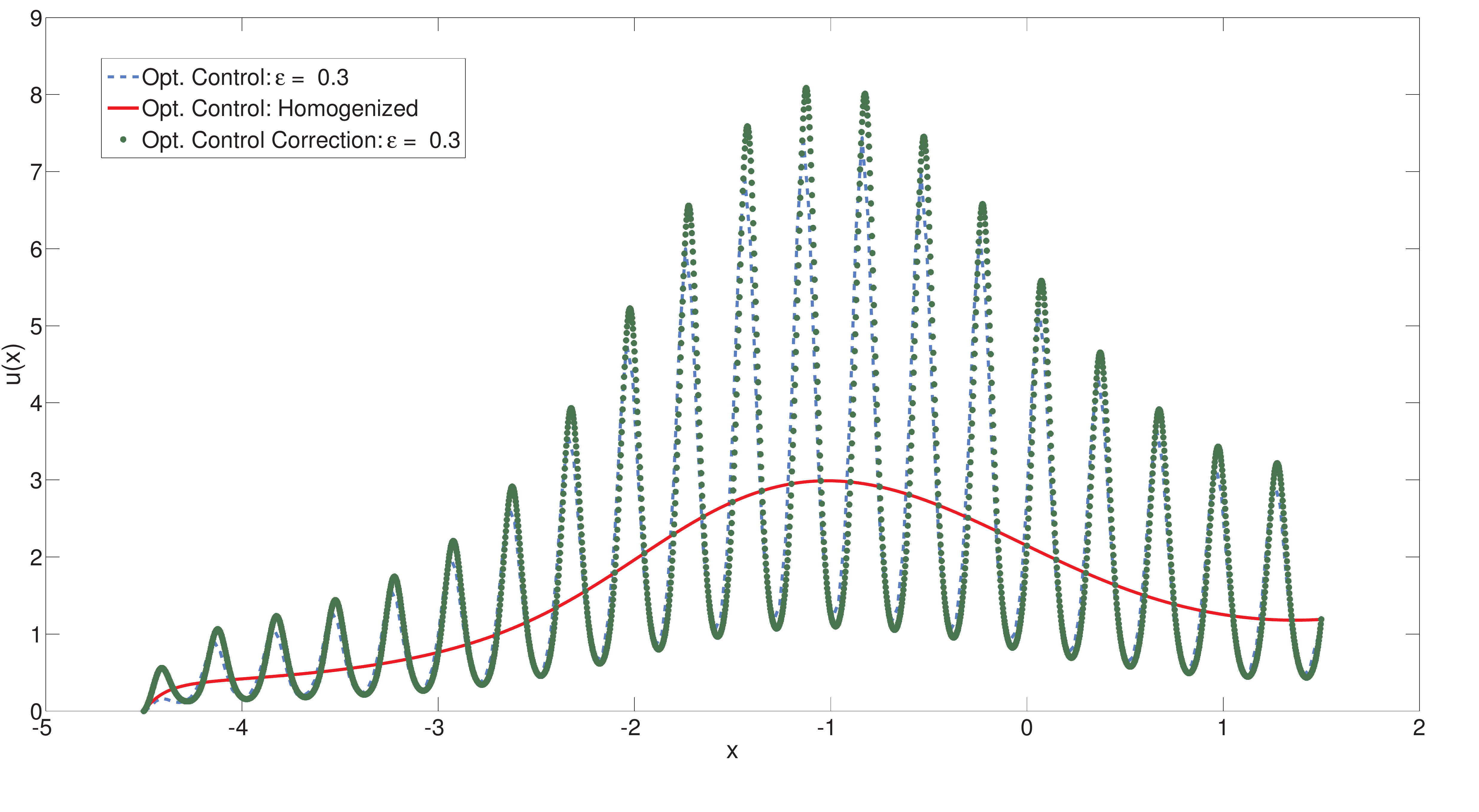}
 \caption{Value function and resulting optimal control (lower panel).} \label{fig:multiscaleControl}
 \end{figure}

\begin{rem}
The above case is an example in which using a reduced-order models for optimal control is not recommended, for $J(\hat{u}^{\eps})$ does not converge to $ J(\hat{u})$ as $\eps\to 0$.  Nonetheless, Theorem~\ref{thm:entropyIndefinite} suggests that we can use the leading term of $c^\epsilon$ in (\ref{ex2-u-2}) as an approximation of the feedback law for the multiscale dynamics (\ref{controled_ex2}). The effect of the corrector estimate (\ref{ex2-u-2}), is to enforce convergence of the \emph{derivative} of the value function, which entails (weak) convergence of the optimal control and convergence  of the optimal cost value (cf.~\cite{is_multiscale} for an application in importance sampling). 
\end{rem}

\begin{figure}
\includegraphics[width=95mm]{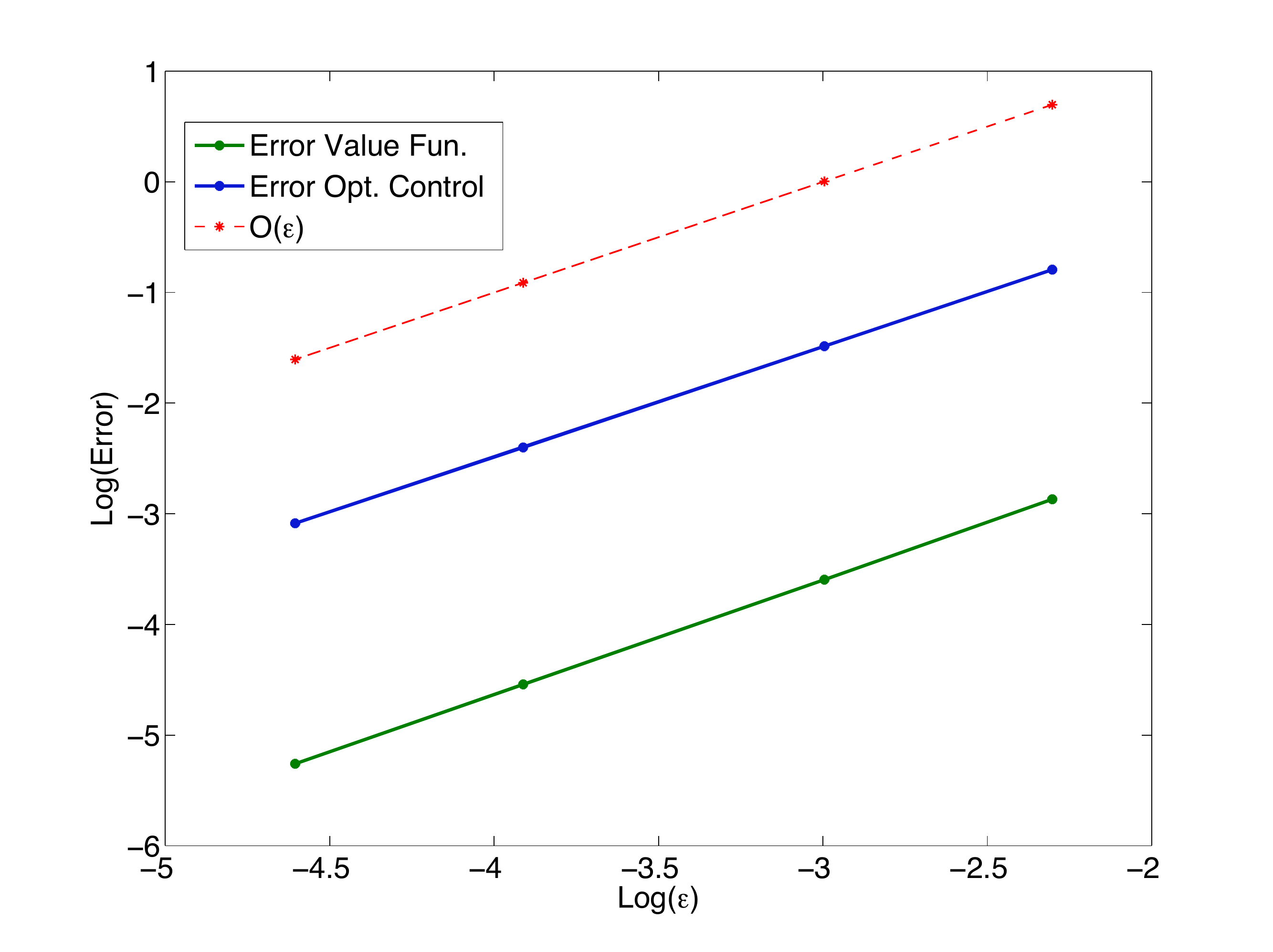}
\caption{Strong $L^{2}$ convergence of value function and optimal control.\label{fig:convVF}}
\end{figure}

\subsection*{Mean first passage time and value function.}
As a specific example, we have solved the optimal control problem
\eqref{eq:valueFuncOneVar}--\eqref{controled_ex2} for the mean first passage time, with $G(x)=1$ and $\tau$ being the first passage time of the set $\{x\le 1.5\}$, and compared it with the solution of the homogenized system \eqref{eq:valueFuncHom}--\eqref{eq:controlledSDEhomo}. The potential $\Phi_{0}$ is chosen to be a tilted double-well potential, 
\begin{displaymath}
 \Phi_0(x)=-5 (\exp{(-0.2(x+2.5)^2)}+\exp{(-0.2(x-2.5)^2)})+0.01x^4+0.8x\,,
\end{displaymath}
with periodic perturbation
\begin{displaymath}
 p(y)=0.5 \sin{2\pi y}\,.
\end{displaymath}
We have solved the associated boundary value problems using the finite-volume method presented in \cite{latorre2011} using a mesh sufficiently fine for the error to be smaller than a certain threshold. The resulting value functions are presented in Figure \ref{fig:vf1}. For comparison, we have also simulated the multiscale system driven by the optimal control for the homogenized system \eqref{eq:optHomogenizedControl},
\begin{equation} \label{eq:micro_wrong}
 dx^{\epsilon,\hat{u}}_s = -\frac{1}{\epsilon}\nabla p(x^{\epsilon,\hat{u}}_s/\epsilon) ds -\nabla \Phi_{0}(x^{\epsilon,\hat{u}}_s)ds + \sqrt{2}\hat{u}_s\, ds + \sqrt{2} \beta^{-1/2} dw_s\,,
\end{equation}
with $\hat{u}_{t}=\hat{c}(x^{\eps,u}_{t})$ and $\hat{c}=-\sqrt{2K}\nabla V_0$. This situation amounts to using the (wrong) homogenized control in the original multiscale dynamics. To illustrate the shortcoming of such an approach, we have calculated the control value 
\begin{displaymath}
J(\hat{u};x)= \bE\left(\tau + \int_{0}^{\tau} \frac{1}{2} |\hat{u}_s|^2 ds~\bigg|~x_0^{\hat{u}} = x\right), 
\end{displaymath}
by Markov-jump Monte Carlo (MJMC) simulations (see \cite{latorre2011}). 
As it is shown in Figure \ref{fig:vf1}, equation \eqref{eq:micro_wrong} does not capture the control value $J(\hat{u}^{\eps})$ as $\eps\to 0$; in order to reproduce the control value correctly, one must instead use the corrected control 
\begin{equation}\label{eq:micro_right}
\tilde{u}_{t}=\tilde{c}(x^{\eps,u}_{t},x^{\eps,u}_{t}/\epsilon)\,,
\end{equation} 
as given in \eqref{ex2-u-2}.

\begin{figure}[h!]
 \centering
 \begin{subfloat}[$\eps=0.1$]{
  \centering
  \includegraphics[width=70mm]{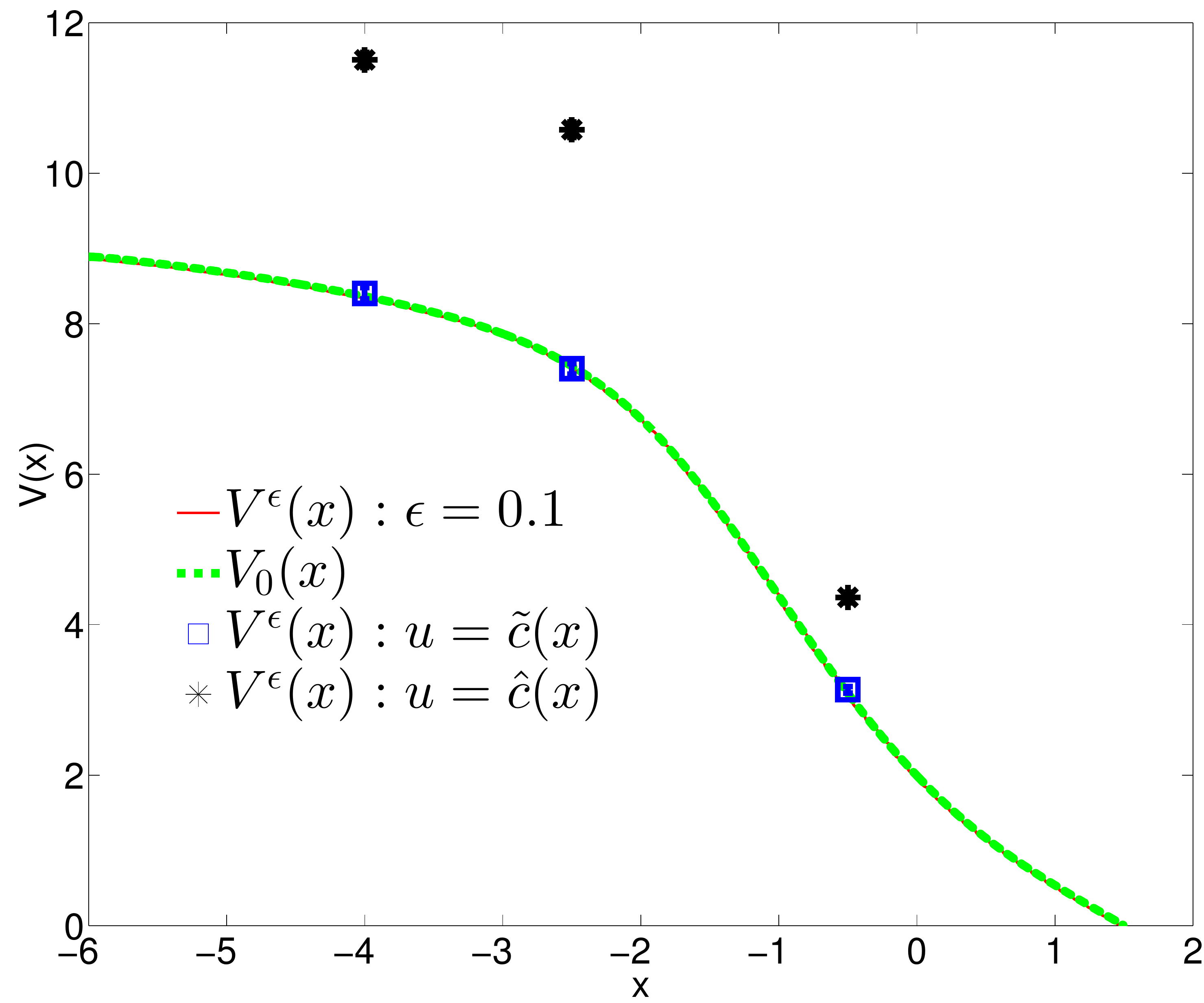}
   \label{fig:vfEps1}}
 \end{subfloat}
 ~ \begin{subfloat}[$\eps=0.05$]{
  \centering
  \includegraphics[width=70mm]{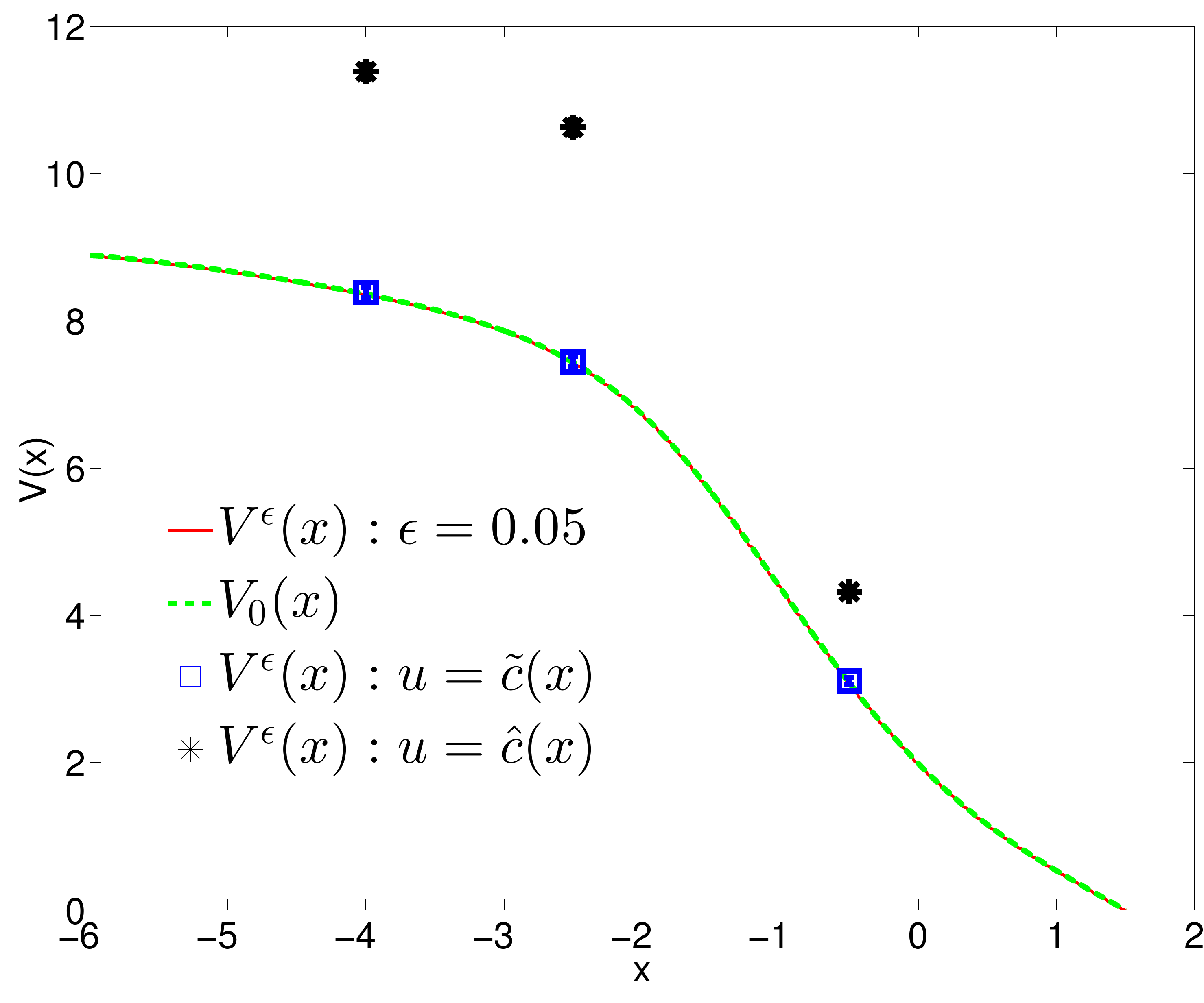}
  \label{fig:vfEps2}}
 \end{subfloat}
 \caption{Optimal value function for different values of $\eps$. Solid line: numerical solution of eq.~\eqref{linbvp}. Dashed line: numerical solution of eq.~\eqref{fk2}. {\color{black} $\star$}: MJMC sampling of \eqref{eq:micro_wrong}. {\color{blue} $\square$}: MJMC sampling using \eqref{eq:micro_right}. Throughout the simulations we have set $\beta=2$} \label{fig:vf1}
\end{figure}

%%%%%%%%%%%%%%%%%%%%%%%%%%%%%%%%%%%%%%%%%%%%
%%%%%%%%%%%%%%%%%%%%%%%%%%%%%%%%%%%%%%%%%%%%
\subsection{Linear-quadratic regulator}\label{subsec-examples-linear}

The third example is a multiscale linear quadratic regulator (LQR) problem that slightly falls out of the previous category. Specifically, we seek to minimize the time-averaged quadratic cost
\begin{equation}\label{LQR1}
J^\eps(u) = \limsup_{T\to\infty}\bE\left(\frac{1}{T}\int_0^{T}
\left\{|x^{\eps,u}_{s}|^{2} + |y^{\eps,u}_{s}|^{2} + \frac{1}{2}|u^\eps_{s}|^{2} \right\}ds
\right)
\end{equation}
subject to the linear dynamics
\begin{equation}\label{LQR2}
\begin{aligned}
dx^{\epsilon,u}_s & = \left(A_{11}x^{\epsilon,u}_s +\frac{1}{\eps}
A_{12}y^{\epsilon,u}_s + \sqrt{2} B_1 u^\eps_s \right) ds + \sqrt{2\beta^{-1}} B_1dw_s \\
dy^{\epsilon,u}_s & = \left(\frac{1}{\eps} A_{21}x^{\epsilon,u}_s +
\frac{1}{\eps^{2}} A_{22}y^{\epsilon,u}_s+ \frac{1}{\eps}\sqrt{2} B_2 u^\eps_{s} \right)ds +
\frac{1}{\epsilon} \sqrt{2} \beta^{-1/2} B_2dw_s
\end{aligned}
\end{equation}
where $x\in \mathbb{R}^k, y \in \mathbb{R}^{n-k}$, $u\in\R^{l}$, and $A_{ij}$, $B_{i}$ are real matrices of appropriate size. Note that both slow and fast equations are driven by the same noise and control. Further let 
\[
A = \left(\begin{array}{cc}
A_{11} & \eps^{-1} A_{12} \\ \eps^{-1} A_{21} & \eps^{-2} A_{22}
\end{array}\right),\quad B
= \sqrt{2}\left(\begin{array}{c}
B_{1}\\\eps^{-1}B_{2}
\end{array}\right).
\]
We make the following additional assumptions (we suppress the $\eps$ in the matrix definition in order to keep  the notation compact):
\begin{enumerate}
\item The initial values $(x^{\eps}_{0},y^{\eps}_{0})=(x_{0},y_{0})$ are independent of $\eps$ and satisfy 
\[
\bE[|x_{0}|^{2}]<\infty\,,\quad \bE[|y_{0}|^{2}]<\infty\,.
\] 
\item For all $\eps>0$, the spectrum of $A$ lies entirely in the open left half complex plane, i.e., all eigenvalues of $A$ have strictly negative real part. 
\item The spectrum of $A_{22}$ lies entirely in the open left half complex plane. 
\item For all $\eps>0$, the matrix pair  $(A,B)$ is controllable, i.e., the matrix
\[
K^{\eps} = (B\; AB\; A^{2}B\;\ldots\;A^{n-1}B^{\eps})
\]
has maximum rank $n$.
\end{enumerate}
For the control problem (\ref{LQR1})--(\ref{LQR2}), the analog of (\ref{hjb}) for the case of an infinite-time horizon with time-averaged cost and unbounded domain reads \cite{robin1983,ScWiHa2012}
\begin{equation}\label{ergodHJB}
\eta^{\eps} = \min_{c\in\R^n}\left\{\cL^{\eps}(c) V^{\eps} + |z|^{2} +
\frac{1}{2}|c|^{2}\right\}
\end{equation}
where $z=(x,y)$ and 
\[
\cL^{\eps}(u) = (2\beta)^{-1}BB^{T}\colon\nabla^{2} + (Az + Bu)\cdot\nabla
\]
The unknown parameter $\eta^{\eps}\in\R$ in the Hamiton-Jacobi-Bellman equation (\ref{ergodHJB}) needs to be determined along with the function $V^{\eps}=V^{\eps}(x,y)$, in fact (\ref{ergodHJB}) can be regarded as a nonlinear eigenvalue equation for the pair $(\eta^{\eps},V^{\eps})$; for details we refer to Appendix \ref{sec:ergode}. 

LQR problems of this kind have quadratic value functions and admit an explicit solution in terms of an algebraic Riccati equation
\begin{equation}\label{ricc}
A^{T}S^{\eps} + S^{\eps}A - 2 S^{\eps}BB^{T}S^{\eps} + I_{n\times n} = 0\,,
\end{equation}
where $I_{n \times n}$ denotes the $n \times n$ identity matrix. Specifically, plugging the ansatz 
\[
V^{\eps}(z) =  z^{T}S^{\eps} z
\]  
into (\ref{ergodHJB}), it readily follows that $S^{\eps}$ solves (\ref{ricc}).
Hence the optimal control for the linear quadratic regulator
(\ref{LQR1})--(\ref{LQR2}) is given by the linear feedback law 
\[
  \hat{u}^{\eps}_{t} = - B^{T}S^{\eps}z_{t}\,.
\]
Under the above assumptions, the Riccati equation has a unique symmetric positive definite solution $S^{\eps}$ for all values of $\eps>0$. Moreover, it follows that  
\[
\eta^{\eps} = BB^{T}\colon S^{\eps}\,,
\] 
which is the principal eigenvalue of the linear eigenvalue equation
\begin{equation}\label{linEVP}
(2\beta)^{-1}BB^{T}\colon\nabla^{2}\psi^{\eps} + (Az) \cdot\nabla\psi^{\eps} -  \beta |z|^{2}\psi^{\eps} = -\beta \eta^{\eps}\psi^{\eps}
\end{equation}
for the log-transformed eigenfunction $\psi^{\eps}=\exp(-\beta V^{\eps})$. Notice that the eigefunction $\psi^{\eps}$ corresponding to the principal eigenvalue $-\beta\eta^{\eps}\le 0$ is strictly positive as a consequence of the Perron-Frobenius theorem, hence its log transformation is well defined.   

\subsection*{Reduced Riccati equation}

Given the above assumptions on the matrices $A$ and $B$, the homogenized version of the linear eigenvalue equation (\ref{linEVP}) can be easily computed, since the cell problem has an explicit solution. We find 
\begin{equation}\label{linEVPred}
(2\beta)^{-1}\bar{B}\bar{B}^{T}\colon\nabla^{2}\psi + (\bar{A}z) \cdot\nabla\psi - \beta (|x|^{2} + Q)\psi = -\beta\eta\psi
\end{equation}
with the homogenized coefficients 
\[
\bar{A} = A_{11} - A_{12}A^{-1}_{22}A_{21}\,,\quad  \bar{B} = \sqrt{2}\left(B_1 + A_{12}A^{-1}_{22}B_2\right) \,
\]
and 
\[
Q = 2\beta^{-1}{\rm tr}\left(\int_{0}^{\infty}e^{A_{22}t}B_{2}B_{2}^{T}e^{A_{22}t}dt\right)\,,
\]
denoting the sum of the eigenvalues of the asymptotic covariance matrix of the fast degrees of freedom.  
The limiting eigenpair $(\eta,\psi)$ is given by 
\[
 \eta = \bar{B}\bar{B}^{T}\colon S + Q\,,\quad \psi(x)=e^{-\beta x^{T}Sx }\,
\]
where $S$ is the solution of the homogenized Riccati equation
\begin{equation}\label{riccred}
\bar{A}^{T}S + S\bar{A} - 2 S\bar{B}\bar{B}^{T}S + I_{k\times k} = 0\,, 
\end{equation}
in accordance with the solution of the algebraic Riccati equation of singularly-perturbed LQR problems that has been discussed in the literature; see \cite{gajic2001} and the references therein. 
%\begin{rem}
It can be shown by perturbation analysis of the Riccati equation (\ref{ricc}) using the Chow transformation (see, e.g., \cite{Kokotovic1987} and the references therein) that $S$ corresponds to the top left $k\times k$ block of the matrix $S$ up to $\cO(\eps^{2})$. Moreover, for any open and bounded subset $\Omega\subset\R^{n}$ with smooth boundary, we have  
\[
\|V^{\eps} - V \|_{H^{1}(\bar{\Omega})} \le C_{1}\eps^{2}\,.
\]
for $V=-\beta^{-1}\log\psi$ and some constant $0<C_{1}<\infty$. The latter implies that
\[ 
  |\hat{u}^{\eps}_{s} - \hat{u}_{s} | \le C_{2}\eps\,
\]
uniformly on $[0,\tau_{\Omega}]$ where $\tau_{\Omega}$ is the first exit time from $\Omega\subset\R^{n}$ and $0<C_{2}<\infty$.  For large values of $\beta$ the probability that the process exits from $\Omega$ is exponentially small in $\beta$, i.e., the exit from the domain is a rare event  (see, e.g., \cite{zabczyk1985}) and hence we can employ the approximation $\tau_{\Omega}\approx\infty$ for all practical purposes.        
%\end{rem}

\subsection*{270-dimensional ISS model} 

We consider the 270-dimensional model of a component of the International Space Station (ISS) that is taken from the SLICOT benchmark library  \cite{chahlaoui2005b}. In this case, $n=270$ and $l=3$ in equation (\ref{LQR1}); the dimension of the slow subspace is set to $k=4$, because the spectrum of dimensionless Hankel singular values of the full system shows a significant spectral gap at $k=4$ when the slow variables are chosen as the observed variables; see \cite{Ha11} for details. The original system is Hamiltonian, but we pay no attention to the specific geometric structure of the equations here; cf.~\cite{balance_hartmann} for related work. The corresponding control task for the 4-dimensional reduced system thus is to minimize
\begin{equation}\label{LQRred1}
\bar{J}(u) = \limsup_{T\to\infty}\bE\left(\frac{1}{T}\int_0^{T}
\left\{|x^{u}_{s}|^{2} + \frac{1}{2}|u_{s}|^{2} \right\}ds
\right)
\end{equation}
subject to the dynamics
\begin{equation}\label{LQRred2}
dx^{u}_s = \left(\bar{A}x^{u}_s + \bar{B} u_s \right) ds + \beta^{-1/2} \bar{B}dw_s\,,
\end{equation}
with $\bar{A}$ and $\bar{B}$ as in (\ref{riccred}).  Without loss of generality, we have ignored the additive constant $Q$
in the cost term that appears in the homogenized eigenvalue equation
(\ref{linEVPred}). As before the optimal control is given by the linear feedback law
\[
  \hat{u}_{s} = -\bar{B}^{T}Sx_{s}\,.
\]   
where $S$ denotes the solution of (\ref{ricc}). To verify the convergence of the value function numerically, we have computed eigenvalues of $S$ and $S^{\eps}$, the matrix norms of $S  - S^{\eps}_{11}$ and the norm of the matrix $S^{\eps}$ with the $S^{\eps}_{11}$ block set to zero, called $S^{\epsilon}_{r}$. Here $S^{\eps}_{11}$ refers to the upper left $k\times k$ block of the matrix $S^{\eps}$, in accordance with the notation in (\ref{LQR2}).  Figure~\ref{fig:lqr} shows this comparison for $\beta=0.01$, which, given the parameters of the ISS model, amounts to the small noise regime; the plots clearly show that the convergence is of $\cO(\eps^{2})$. 
We refrain from testing the convergence  $\eta^{\eps}\to\eta$ of the corresponding nonlinear eigenvalue since the $1/\eps^{2}$ singularity makes the evaluation of the trace term $BB^{T}\colon S^{\eps}$ numerically unstable for all interesting values of $\eps$.

\begin{figure}[h!]
 \begin{center}
  \includegraphics[width=0.495\textwidth]{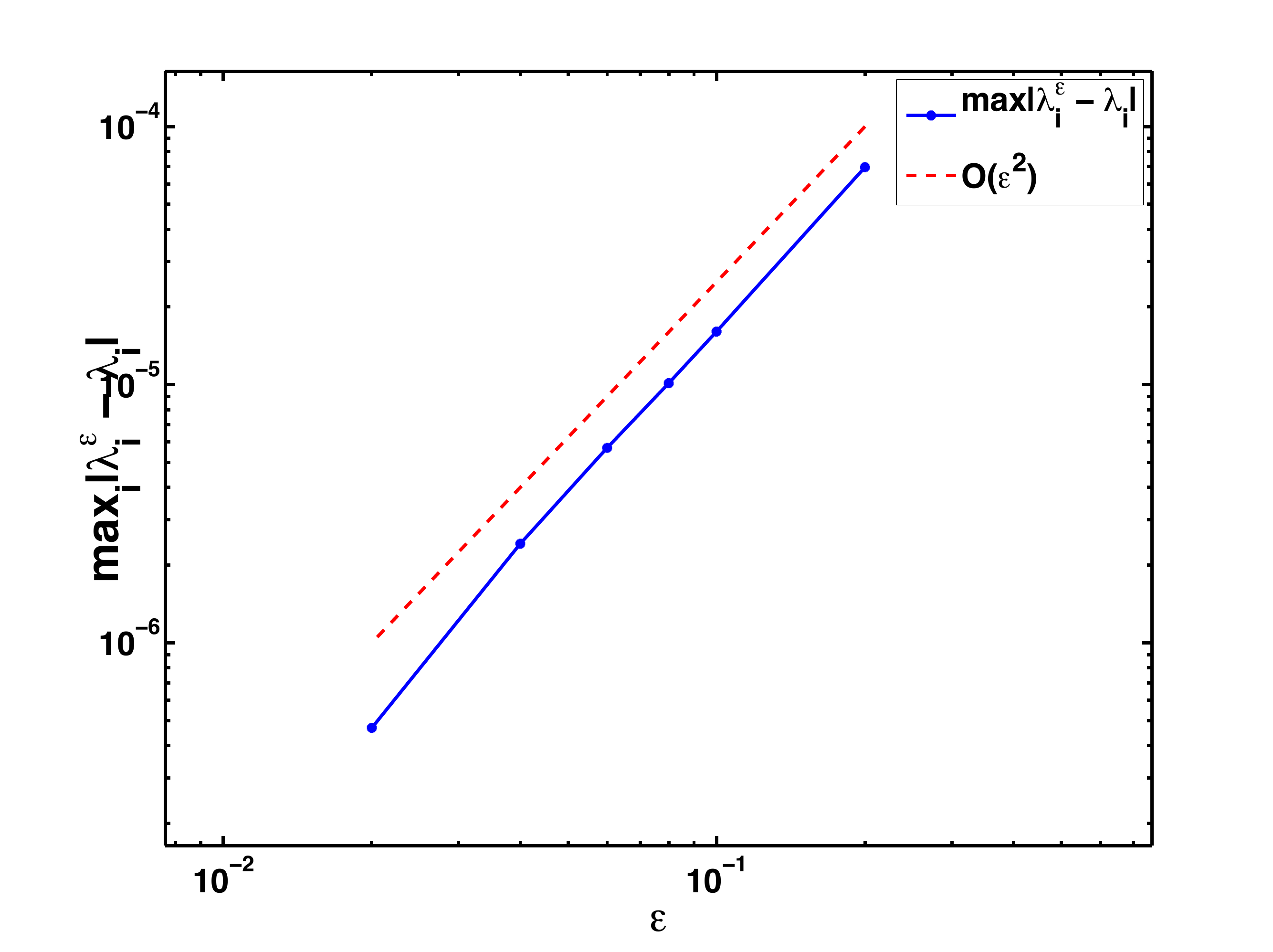}
  \includegraphics[width=0.495\textwidth]{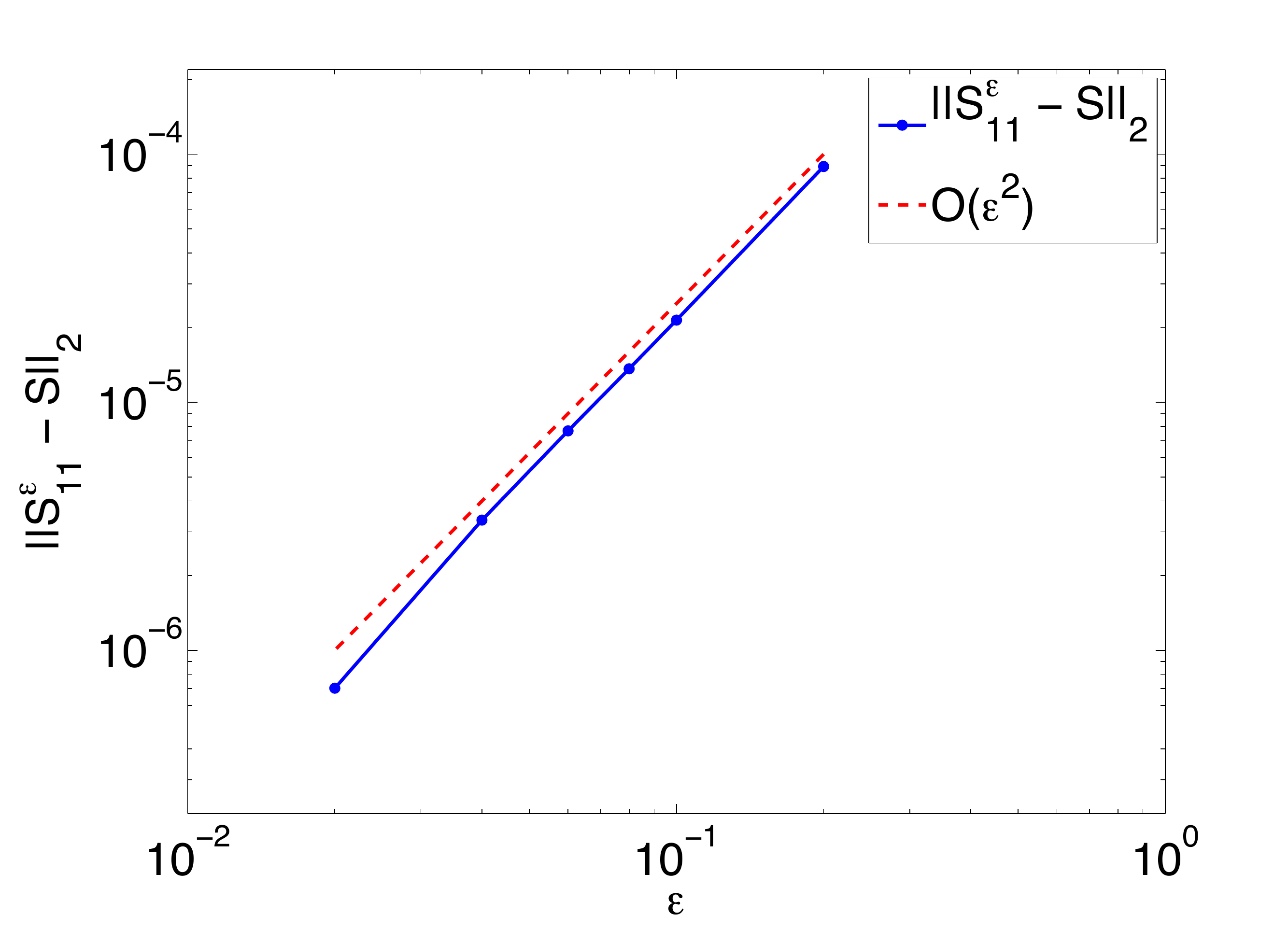}\\
  \includegraphics[width=0.495\textwidth]{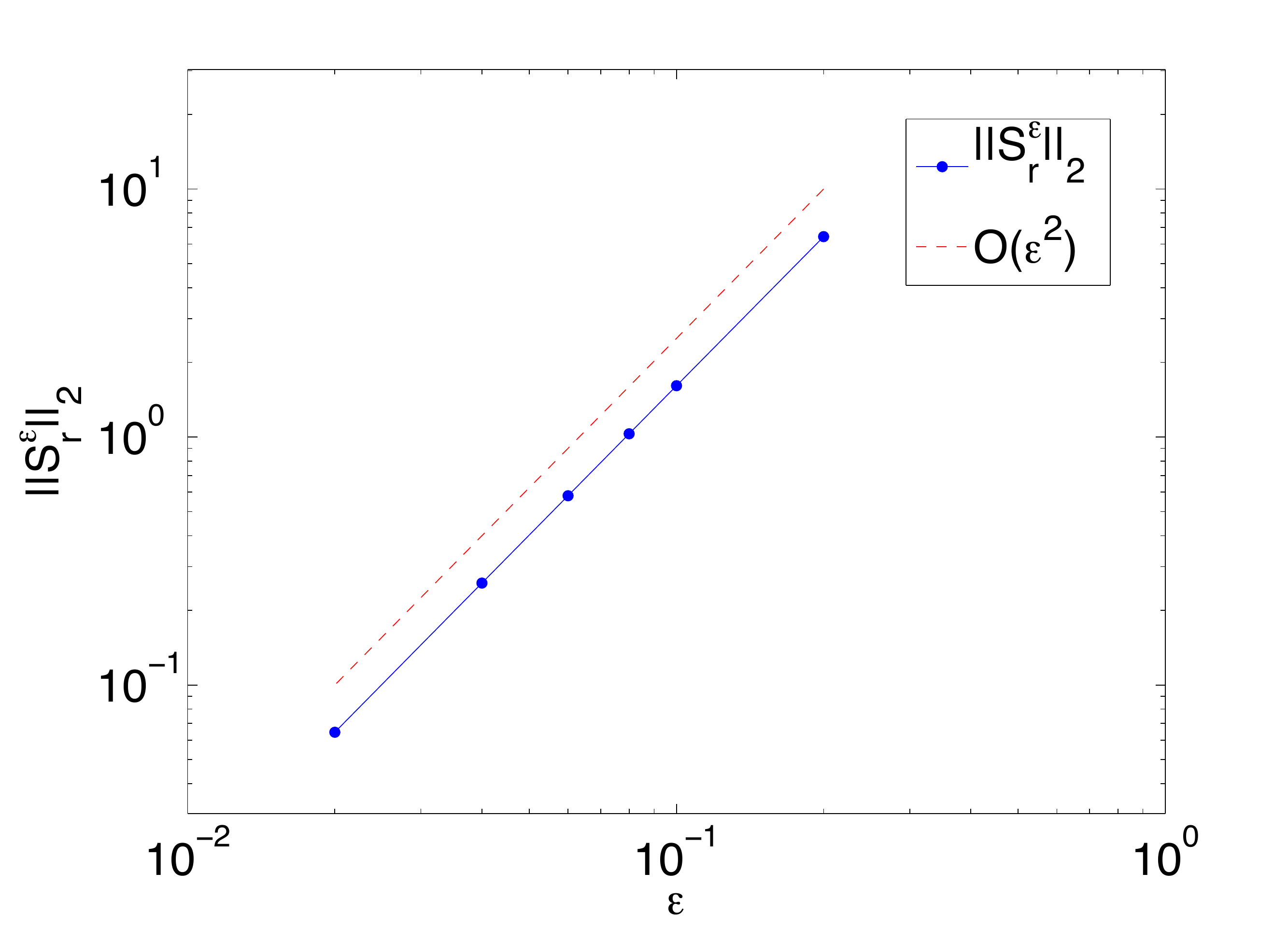}  
     \includegraphics[width=0.495\textwidth]{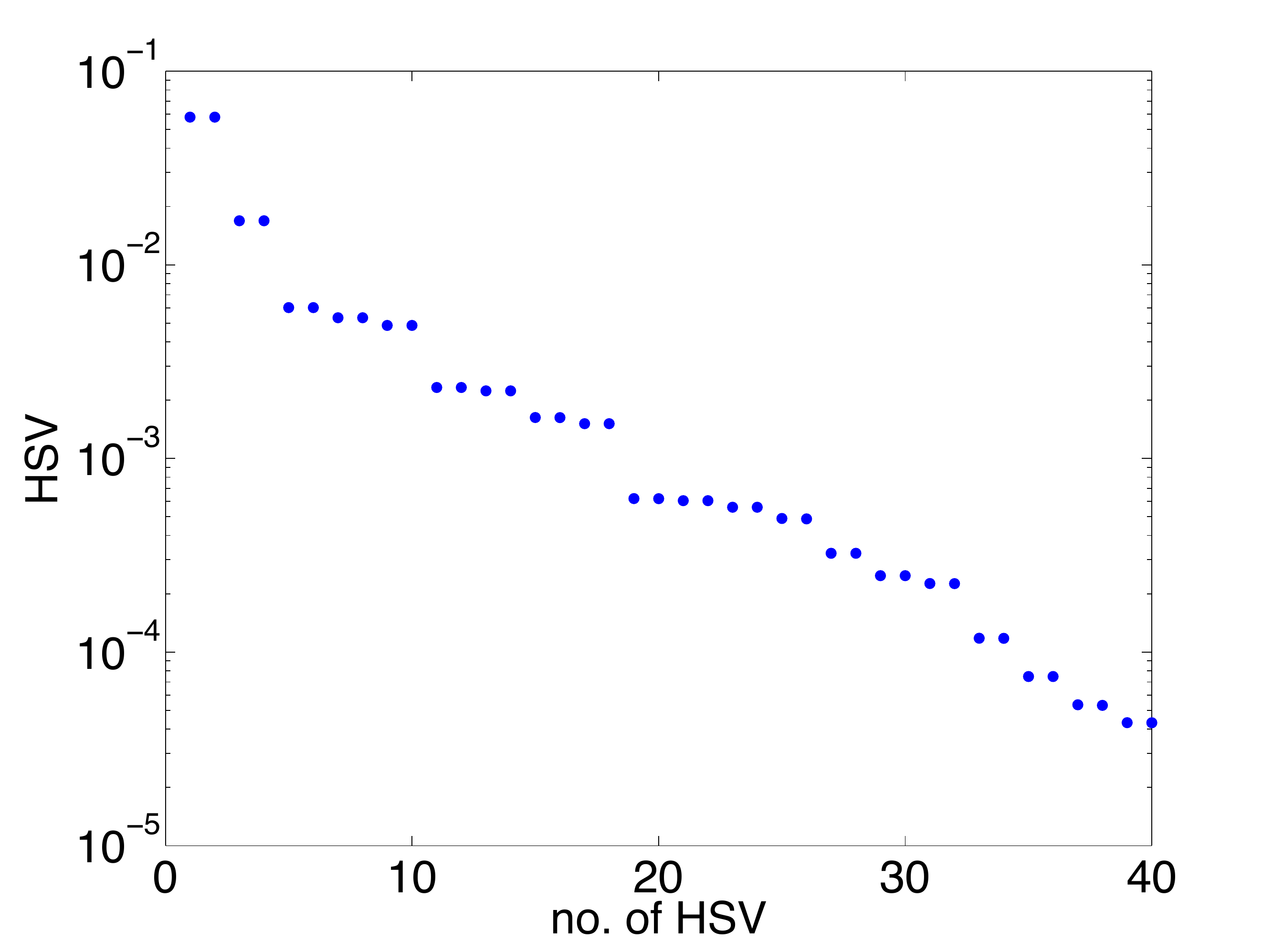}
  \caption{Hankel singular values and quadratic convergence of the matrix $S^{\eps}$ in terms of the $k$ dominant eigenvalues (upper left panel), the 1-1 matrix block (upper right panel) and the residual matrix $S^{\eps}_{r}$ (lower left panel); for smaller values of $\eps$ the numerical solution of the Riccati equation is dominated by round-off errors, hence the results are not shown. The lower right panel shows the first 40 Hankel singular values  (out of 270) when the slow variables are observed; the Hankel singular values are independent of $\eps$. }
   \label{fig:lqr}
\end{center}
\end{figure}

%%%%%%%%%%%%%%%%%%%%%%%%%%%%%%%%%%%%%%%%%%%%
%% APPENDIX
%%%%%%%%%%%%%%%%%%%%%%%%%%%%%%%%%%%%%%%%%%%%

\appendix

\section{Weak convergence under logarithmic transformations}\label{sec:weak}

As we have seen in Section \ref{subsec-examples-homo} loss of backward stability of the model reduction approach 
is related to weak convergence of the multiscale controls. Weak convergence is mainly an issue for homogenization problems with 
periodic coefficients that do not involve any explicit time-dependence. For control problems on a finite time-horizon, a well-known result (e.g., see \cite[Sec.~3]{bensoussan1978} or \cite[Sec.~20]{pavliotis2008}) that is based on the maximum principle states that the convergence of the log-transformed parabolic equation is uniform on bounded time intervals under fairly weak assumptions. 

In the indefinite time-horizon case considered in this paper, however, the lowest order approximation gives only weak convergence. 
In general, weak convergence is not preserved under nonlinear transformation. That is, given a weakly convergent sequence $\psi^{\eps}$ on $\R$ and a nonlinear continuous function $F\colon\R\to\R$, we have  
\[
\psi^{\eps}\wto \psi \quad \not\Rightarrow \quad F(\psi^{\eps})\wto F(\psi)\,.
\]
In our case, however, weak convergence follows from the properties of the logarithm and the fact that $\psi^{\eps}$ is bounded away from 0. 
Let $\psi^{\eps}$ be the solution of the elliptic boundary value problem (\ref{linbvp}) for $T\to\infty$ and recall that 
\begin{equation*}
\psi^{\eps}\to\psi\quad\textrm{strongly in}\quad L^{2}(\bar{O})
\end{equation*}
and 
\begin{equation*}
\psi^{\eps}\wto\psi\quad\textrm{weakly in}\quad H^{1}(\bar{O})\,.
\end{equation*}
Moreover, we have that
\begin{equation*}
0<C\le \psi^{\eps}\le 1 \quad  \eps >0 
\end{equation*}
for some $C\in(0,1)$. 

\begin{lem} We have
\[
\log\psi^{\eps}\to \log\psi \quad\textrm{strongly in}\quad L^{2}(\bar{O})
\]
\end{lem}

\begin{proof}
Since $C\le \psi^{\eps}\le 1$ the monotony of the logarithm entails that 
\[
\log C \le \log\psi^{\eps} \le 0 \,.
\]
Since $\log C>-\infty$ and $O\subset\R^{n}$ is bounded it follows that $\log\psi^{\eps}\in L^{2}(\bar{O})$ and, by the same argument,  $\log\psi\in L^{2}(\bar{O})$. Convergence now follows from the fact that $\log(x)$ is Lipschitz continuous with a Lipschitz constant $L<\infty$ if $x\le C>0$: 
\begin{align*}
\| \log\psi^{\eps} - \log\psi\|^{2}_{L^{2}(\bar{O})} & = \int_{\bar{O}} |\log\psi^{\eps} - \log\psi|^{2} dz\\
& \le  L^{2}\int_{\bar{O}} |\psi^{\eps} - \psi|^{2} dz\,,
\end{align*}
which vanishes in the limit $\eps\to 0$ as $\psi^{\eps}\to\psi$ in $L^{2}(\bar{O})$. 
\end{proof}

This implies strong convergence of the value function. For the optimal control, the above conditions give only weak convergence, which is implied by: 

\begin{lem} We have
\[
\log\psi^{\eps}\wto \log\psi \quad\textrm{weakly in}\quad H^{1}(\bar{O})
\]
\end{lem}

\begin{proof}
It suffices to show that $\nabla\log\psi^{\eps}\wto\nabla\log\psi$ in $L^{2}(\bar{O})$. To this end recall that $\nabla\psi^{\eps}\wto\nabla\psi$ in $L^{2}(\bar{O})$ since $\psi^{\eps}$ converges weakly in $H^{1}(\bar{O})$. Then, for all test functions $\phi\in L^{2}(\bar{O})$, using again that $\psi^{\eps}\ge C>0$ pointwise and uniformly in $\eps$, 
\begin{align*}
%\int_{\bar{O}} \left(\nabla\log\psi^{\eps} - \nabla\log\psi\right)\Phi \,dz & = 
\int_{\bar{O}} \left(\frac{\nabla\psi^{\eps}}{\psi^{\eps}} - \frac{\nabla\psi}{\psi}\right)\phi \,dz 
& = \int_{\bar{O}} \left(\psi \nabla\psi^{\eps}- \psi^{\eps}\nabla\psi \right)\frac{\phi}{\psi^{\eps}\psi} \,dz\\
& \le  \frac{1}{C^{2}} \int_{\bar{O}}  \left(\psi \nabla\psi^{\eps}- \psi^{\eps}\nabla\psi \right) \phi  \,dz\\
& \le  \frac{1}{C^{2}} \underbrace{\int_{\bar{O}}  \left(\psi \nabla\psi^{\eps}- \psi\nabla\psi \right) \phi  \,dz}_{I_{1}} + \frac{1}{C^{2}} \underbrace{\int_{\bar{O}}  \left(\psi\nabla\psi- \psi^{\eps}\nabla\psi \right) \phi  \,dz}_{I_{2}}
\end{align*}
We look at the two integrals separately. Using that $0<\psi\le 1$ it follows that 
\[
|I_{1} | \le \left|\int_{\bar{O}}  \left(\nabla\psi^{\eps}- \nabla\psi \right) \phi  \,dz \right| \to 0
\]
since $\phi\in L^{2}(\bar{O})$ and $\nabla\psi^{\eps}\wto\nabla\psi$ weakly in $L^{2}(\bar{O})$. Now for the second integral: since the weakly convergent sequence $\psi^{\eps}$ and its limit $\psi$ are bounded in $H^{1}(\bar{O})$ we conclude that $\nabla\psi\in L^{2}(\bar{O})$, which together with the boundedness of $|\psi^{\eps}-\psi|$ implies that $(\psi^{\eps}-\psi)\nabla\psi\in L^{2}(\bar{O})$. So, by the Cauchy-Schwarz inequality,  
\begin{align*}
|I_{2}|^{2} & \le \left(\int_{\bar{O}}  |(\psi^{\eps}- \psi)\nabla\psi|^{2}  \,dz\right)\left( \int_{\bar{O}}  |\phi|^{2}  \,dz\right)\\
& = \|\phi \|^{2}_{L^{2}(\bar{O})} \int_{\bar{O}}  |(\psi^{\eps}- \psi)\nabla\psi|^{2}  \,dz \\
& \le M\|\phi \|^{2}_{L^{2}(\bar{O})}\int_{\bar{O}}  |(\psi^{\eps} - \psi)\nabla\psi| \,dz\\
\end{align*}
for some constant $0<M<\infty$. Reiterating the preceding argument it follows that  
\[ 
|I_{2}|^{2} \le M\|\phi \|^{2}_{L^{2}(\bar{O})}\|\psi^{\eps}-\psi \|^{2}_{L^{2}(\bar{O})}\|\nabla\psi\|^{2}_{L^{2}(\bar{O})}\to 0
\]
as $\psi^{\eps}\to\psi$ in $L^{2}(\bar{O})$ and $\nabla\psi\in L^{2}(\bar{O})$. Hence 
\[
\left|\int_{\bar{O}} \left(\nabla\log\psi^{\eps} - \nabla\log\psi\right)\phi \,dz \right| \to 0
\]
which, together with the last Lemma yields the assertion.  
\end{proof}

%%%%%%%%%%%%%%%%%%%%%%%%%%%%%%%%%%%%%%%%%%%%
%% ADDENDUM (ergodic control)
%%%%%%%%%%%%%%%%%%%%%%%%%%%%%%%%%%%%%%%%%%%%

\section{Ergodic control problem}\label{sec:ergode}

We briefly discuss the ergodic control problem of Section \ref{subsec-examples-linear} that is known to be related to an elliptic eigenvalue problem \cite{holland1978,nagai1990,fleming1995}. In principle, the equivalence of (\ref{linEVP}) and (\ref{ergodHJB}) directly follows from the logarithmic transformation. 
Here, we give an alternative derivation of the associated HJB equation, starting from the underlying Kolmogorov backward equation. To this end let
\begin{equation}\label{ergodic}
\eta^{\eps} = -\limsup_{T\to\infty}\frac{1}{\beta T}\log\bE\left(\exp\left(-\beta \int_{0}^{T} G(z^{\eps}_{t})\,dt\right)\right)\,.
\end{equation}
for a continuous bounded function $G\colon\R^{n}\to [0,\infty)$ 
Further let $\varphi(z,t)$ be given by 
\begin{equation}\label{varphi}
\varphi^{\eps}(z,t) = \bE\left(\exp\left(-\beta \int_{0}^{t} G(z^{\eps}_{s})\,ds\right)\bigg|\,z^{\eps}_{0}=z\right).
\end{equation}
By the Feynman-Kac formula $\varphi^{\eps}(z,t)$ is the solution of
\begin{equation}\label{FK3}
\begin{aligned}
\left(\frac{\partial}{\partial t} - \cL^{\eps}\right)\varphi^{\eps} & = -\beta G\varphi^{\eps}\\ \varphi^{\eps}(z,0) & = 1\,.
\end{aligned}
\end{equation}
Here
\[
\cL^{\eps} = \frac{1}{2} \beta^{-1} \sigma(z;\eps)\sigma(z;\eps)\colon\nabla^{2} + b(z;\eps)\cdot\nabla
\]
denotes the infinitesimal generator of our generic uncontrolled diffusion process. Setting $V^{\eps}=-\beta^{-1}\log\varphi^{\eps}$, we can rewrite Equation~\eqref{ergodic} in the form 
\begin{equation*}
\eta^{\eps} = \lim_{t\to\infty}\frac{V^{\eps}(z,t)}{t}\,.
\end{equation*}
Assuming that the limit exists, this motivates the following asymptotic ansatz for large $t$:
\begin{equation*}
\varphi^{\eps}(z,t) \sim \psi^{\eps}(z)\exp(-\eta^{\eps}\beta t)\,,\quad \psi^{\eps}>0\,.
\end{equation*}
Plugging the separation ansatz into (\ref{FK3}) it follows that $\psi^{\eps}$ solves the eigenvalue equation
\[
\left(G - \beta^{-1}\cL^{\eps} \right)\psi^{\eps} = \eta^{\eps} \psi^{\eps}\,,
\]
or, equivalently,
\[
\left(\cL^{\eps} - \beta G \right)\psi^{\eps} = -\beta\eta^{\eps} \psi^{\eps}\,,
\]
As a consequence of the Perron-Frobenius theorem the eigenfunction $\psi^{\eps}$ corresponding to the principal eigenvalue 
$-\beta\eta^{\eps}$ is strictly positive. The equivalent nonlinear eigenvalue problem for the log-transformed eigenfunction $V^{\eps}=-\beta^{-1}\log \psi^{\eps}$ reads    
\[
\cL^{\eps} V^{\eps} - \frac{1}{2}|\sigma^{T}\nabla V^{\eps}|^{2} + G = \eta^{\eps}\,.
\]
which, as before, can be rewritten in the form 
\[
\min_{c\in\R^{n}}\left\{ (\cL^{\eps} V^{\eps} + (\sigma c)\cdot\nabla V^{\eps} + G + \frac{1}{2}|c|^{2} \right\} = \eta^{\eps}\,.
\]
The last equation is recognized as the dynamic programming equation of the ergodic optimal control problem, of which (\ref{LQR1})--(\ref{LQR2}) is a special case: minimize
\[
J^\eps(u) = \limsup_{T\to\infty}\bE\left(\frac{1}{T}\int_{0}^{T}\left(G(z^{\eps}_{s}) + \frac{1}{2}|u_{s}|^{2}\right)ds\right)
\]
subject to 
\[
dz^{\eps,u}_{s} = \left(b(z^{\eps,u}_{s};\eps) + \sigma(z^{\eps,u}_{s};\eps)u^\eps_{s}\right)ds + \sigma(z^{\eps,u}_{s};\eps) \beta^{-1/2} dW_{s}\,. 
\]

\subsection{Homogenized ergodic control problem} 
Let $z=(x,y)$ and consider the expansion $\psi^\epsilon = \psi_0 + \epsilon\psi_1 + \cdots$ and $\eta^\epsilon = \eta_{0} + \epsilon \eta_1 + \cdots$, as in the previous subsections. The leading term in the expansion $\psi_0$ is independent of $y$ and satisfies
\begin{align*}
(\bar{\cL} - \beta \bar{G})\psi_0= -\beta\eta_0\psi_0\,,
\end{align*}
with $\bar{\cL}, \bar{G}$  defined in~\eqref{homo-generator2}. Now suppose $V^\epsilon = V_0 + \epsilon V_1 + \cdots$, then again 
\[
  V_0 = -\beta^{-1}\log \psi_0, \quad V_1 = -\beta\frac{\psi_1}{\psi_0}.
\]
This indicates that the leading nonlinear eigenpair $(\eta_0, V_0)$ satisfies 
\[
\eta_0 = \limsup_{T \to\infty} \bE\left(\frac{1}{T}\int_0^{T} \left(\bar{G}(x_s) + \frac{1}{2}
|\bar{\alpha}(x_{s})^{T}\nabla V_{0}(x_{s})|^2\right) ds\right)\,,
\]
where $x_{s}$ solves the optimally controlled SDE
\[
dx_{s} = \left(\bar{f}(x_{s}) - \bar{\alpha}(x_{s})\bar{\alpha}(x_{s})^{T}\nabla V_{0}(x_{s})\right)ds + \bar{\alpha}(x_{s}) \beta^{-1/2} dw_{s}\,.
\]
By ergodicity of the controlled process, the above expectation is independent of the distribution of the initial values; see \cite{ScWiHa2012} and the references therein.

%%%%%%%%%%%%%%%%%%%%%%%%%%%%%%%%%%%%%%%%%%%%
%% ADDENDUM (entropy bounds)
%%%%%%%%%%%%%%%%%%%%%%%%%%%%%%%%%%%%%%%%%%%%

\section{Entropy bounds for the cost function}\label{sec:bounds}
In this section we study the cost function of the optimal control problem
from the point of view of change of measure.
\label{sec-cost_function}
Consider the SDE
\begin{align}
\begin{split}
  dz_s =&\, b(z_s)\, ds +\beta^{-1/2} \sigma(z_s)\, dw_s  \\
   z_0 =&\, z 
\end{split}
\label{sec-cost-sde1}
\end{align}
and the controlled SDE
\begin{align}
\begin{split}
  dz_s =&\, (b(z_s) + \sigma(z_s)u_{s})\,ds + \beta^{-1/2}\sigma(z_s)\, dw_s \\
   z_0 =&\, z, 
\end{split}\label{sec-cost-csde}
\end{align}
where $u_{s}$ is any bounded measurable control that is adapted to $z_{s}$. Let $\mu$ and $\mu_u$ denote the path
measures generated by (\ref{sec-cost-sde1}) and (\ref{sec-cost-csde}), respectively. 
Then by Girsanov's theorem \cite{oksendal2003}, we have that
\begin{equation}
  \frac{d\mu_u}{d\mu} = \exp\left(-\beta^{1/2}\int_0^\tau u_{s}\, dw_s -
\frac{\beta}{2}\int_0^\tau |u_{s}|^2\, ds\right).\label{sec-cost-girsanov}
\end{equation}
Let a cost functional be given by
\begin{equation}
J(u) = \bE_{\mu_u}\left(\int_0^\tau \left(G(z_s) + \frac{1}{2} |u_{s}|^2
\right) ds ~\bigg|~ z_0 = z\right),
\end{equation}
where $G$ satisfies Assumption 2 from Section~\ref{ssec:logtrafo}. Here we use the notation $\bE_{\mu_u}$ to indicate that the expectation is understood \wrt the probability measure $\mu_u$. Moreover the dependence of $J$ on the initial value $z$ is omitted. 

Let $\hat{u} = \argmin J(u)$, then from Theorem \ref{thm:verification} we know $\hat{u}_{s}$ only depends on $z_{s}$.
Let $\hat{\mu}$ denote the measure $\mu_{\hat{u}}$ for simplicity.
Our purpose here is to estimate $|J(u) - J(\hat{u})|$ when $||u - \hat{u}||_{L^\infty}$ is small. We will make use of the following definition. 

\begin{defn}
For two probability measures $\mu_u,\mu$ with $\mu_{u}\ll\mu$, the Kullback-Leibler
divergence of $\mu_{u}$ relative to $\hat{\mu}$ is defined as 
\begin{equation}
  I(\mu_u\,|\,\hat{\mu}) = \int \log\left(\frac{d\mu_u}{d\hat{\mu}}\right) \,d\mu_u.
\end{equation}
\end{defn}
We also assume that Assumption 3 from Section~\ref{ssec:logtrafo} holds: there exists $\gamma > 0$, such that $
\bE_{\mu}(e^{\gamma \tau}) = C_1 < +\infty$. As in Section \ref{ssec:logtrafo}, we
have that 
\[
  \bE_{\mu}\left(\exp\Big(-\beta\int_0^\tau G(z_s)\, ds\Big)\right) \ge C_1^{-\beta M_1/\gamma} 
\]
Here and in the following, the conditioning on the initial value is omitted.

We also need two technical estimates in order to study the convergence of the cost functional. We start with the following estimate.
\begin{lem}
\label{lemma-0}
$\bE_{\hat{\mu}} (e^{\gamma \tau}) \le C_1^{1+\beta M_1/\gamma}.$
\end{lem}
\begin{proof}
  we have 
    $\bE_{\hat{\mu}} (e^{\gamma \tau}) = \bE_{\mu}(e^{\gamma \tau}
    \frac{d\hat{\mu}}{d\mu})$. Using the dual relation
    \begin{align*}
      -\beta^{-1}\log \bE_\mu\left(\exp\Big(-\beta\int_0^\tau G(z_s)\, ds\Big)\right) = \inf_u J(u) = J(\hat{u})
    \end{align*}
    and Jensen's inequality, we know that 
    \begin{align}
    \exp\Big(-\beta\int_0^\tau G(z_s)\,ds\Big)\frac{d\mu}{d\hat{\mu}} = \bE_\mu\left(\exp\Big(-\beta\int_0^\tau G(z_s)\, ds\Big)\right) \ge C_1^{-\beta M_1/\gamma}  , \hspace{3mm} \mu-a.s.
    \label{jensen-as}
  \end{align}
  where we have assumed the equivalence of $\mu$ and $\hat{\mu}$. Since $G$ is nonnegative, 
  \begin{align*}
    \bE_{\hat{\mu}} (e^{\gamma \tau}) = \bE_{\mu}\left(e^{\gamma \tau}
    \frac{d\hat{\mu}}{d\mu}\right) \le C_1^{\beta M_1/\gamma} \bE_{\mu}(e^{\gamma
    \tau}) = C_1^{1+\beta M_1/\gamma}.
\end{align*}
\end{proof}

  The following lemma provides us with an estimate on the relative entropy when the control $u$ is close to  $\hat{u}$.
\begin{lem}
\label{lemma-1}
Suppose there is an $\epsilon > 0$, such that $|u_{s} - \hat{u}_{s}| \le \epsilon$,
for all $s>0$, and let $\epsilon < (\gamma/\beta)^{1/2}$. Then 
\[
  I(\mu_u\,|\,\hat{\mu}) \le \beta C_3\epsilon^2\,,\quad \bE_{\mu_u}(\tau) \le 2C_3\,,
\]
with the constant $C_3 = \gamma^{-1} (1 + \beta M_1/\gamma) \log C_1$.
\end{lem}
\begin{proof}
From (\ref{sec-cost-girsanov}), we know 
\begin{align*}
  I(\mu_u\,|\, \hat{\mu}) 
  = \int \log\left(\frac{d\mu_u}{d\hat{\mu}}\right) d\mu_u
  = \frac{\beta}{2}~\bE_{\mu_u} \left(\int_0^\tau |u_{s}-\hat{u}_{s}|^2
ds\right) 
\le \frac{\beta}{2} \epsilon^2 \bE_{\mu_u}(\tau).
%\label{lem-1-eqn1}
\end{align*}
On the other hand, by Jensen's inequality, 
\begin{align*}
  \log \bE_{\hat{\mu}} (e^{\gamma \tau}) \ge \gamma \bE_{\mu_{u}}(\tau) -
  I(\mu_u\,|\, \hat{u}).
%\label{lem-1-eqn2}
\end{align*}
The conclusion follows from the last two inequalities.
\end{proof}

Now we are ready to prove Theorem~\ref{thm:entropyIndefinite}, which is restated here more precisely.
\begin{thm}
  Let Assumption 1,2 and 3 from Section \ref{ssec:logtrafo} hold.  
  Further suppose that $\eps < (\gamma/\beta)^{1/2}$ and
  $|u_{s}-\hat{u}_{s}| \le \epsilon$, for all $s>0$. Then it holds that 
\begin{equation}
  J(u) = J(\hat{u}) + \beta^{-1}I(\mu_u|\, \hat{\mu}) \le J(\hat{u}) + C_3 \epsilon^2.
\end{equation}
%\label{thm:entropyIndefinite}
\end{thm}
\begin{proof}
  \begin{align*}
    J(u) = \bE_{\hat{\mu}} \left\{\Big[\int_0^\tau \Big(G(z_s) + \frac{1}{2} |u_{s}|^2\Big)\,ds\Big]\frac{d\mu_u}{d\hat{\mu}}\right\}  \\
\end{align*}
It follows from (\ref{jensen-as}) that we can write the above as 
\begin{align}
  J(u) = J(\hat{u}) + \bE_{\hat{\mu}}\Big[\Big(\beta^{-1}\log
    \frac{d\mu}{d\hat{\mu}} + \int_0^\tau \frac{1}{2} |u_{s}|^2\,ds\Big)\frac{d\mu_u}{d\hat{\mu}}\Big]
\end{align}
Combining this with (\ref{sec-cost-girsanov}), we get 
\begin{align*}
J(u) = J(\hat{u}) + \beta^{-1}I(\mu_u|\, \hat{\mu}) \,.
\end{align*}
The conclusion now readily follows from Lemma \ref{lemma-1}.
\end{proof}

\bibliographystyle{siam}

\end{document}